\documentclass{amsart}
\usepackage[dvips]{graphicx}
\usepackage{enumerate}
\usepackage{amscd}
\usepackage{amssymb}

\title[Ending Lamination Conjecture]{Geometric approach to Ending Lamination Conjecture}
\author{Teruhiko Soma}
\subjclass[2000]{Primary 57M50; Secondary 30F40}
\keywords{Hyperbolic 3-manifolds, Ending Lamination Conjecture, curve graphs}
\thanks{The third version (\today)}
\address{Department of Mathematics and Information Sciences,
Tokyo Metropolitan University,
Minami-Ohsawa 1-1, Hachioji, Tokyo 192-0397, Japan}

\email{tsoma@tmu.ac.jp}

\begin{document}

\begin{abstract}
We present a new proof of the bi-Lipschitz model theorem, which occupies the main part of  
the Ending Lamination Conjecture proved by Minsky \cite{mi2} and Brock, Canary and Minsky \cite{bcm}.
Our proof is done by using techniques of standard hyperbolic geometry as much as possible.
\end{abstract}

\maketitle

\newtheorem{theorem}{Theorem}[section]
\newtheorem{cor}[theorem]{Corollary}
\newtheorem{lemma}[theorem]{Lemma}
\newtheorem{prop}[theorem]{Proposition}

\theoremstyle{definition}
\newtheorem{definition}[theorem]{Definition}
\newtheorem{example}[theorem]{Example}
\newtheorem{remark}[theorem]{Remark}

\numberwithin{figure}{section}
\numberwithin{equation}{section}

\def\hh{\mathbf{H}}
\def\zz{\mathbf{Z}}
\def\nn{\mathbf{N}}
\def\rr{\mathbf{R}}
\def\bt{\mathbf{T}}

\def\bbt{\mathbb{T}}

\def\ca{\mathcal{A}}
\def\cc{\mathcal{C}}
\def\cb{\mathcal{B}}
\def\ck{\mathcal{K}}
\def\cf{\mathcal{F}}
\def\cg{\mathcal{G}}
\def\ch{\mathcal{H}}
\def\cp{\mathcal{P}}
\def\cn{\mathcal{N}}
\def\cs{\mathcal{S}}
\def\cu{\mathcal{U}}
\def\cv{\mathcal{V}}
\def\cy{\mathcal{Y}}

\def\eset{\emptyset}
\def\part{\partial}
\def\to{\longrightarrow}
\def\sto{\rightarrow}
\def\sm{\setminus}
\def\wh{\widehat}
\def\wt{\widetilde}
\def\ol{\overline}
\def\fd{\pi_1}
\def\sg{\sigma}

\def\dist{\mathrm{dist}}
\def\diam{\mathrm{diam}}
\def\Int{\mathrm{Int}}
\def\fr{\mathrm{Fr}}
\def\area{\mathrm{Area}}
\def\ve{\varepsilon}
\def\L{\Lambda}
\def\O{\Omega}
\def\G{\Gamma}
\def\D{\Delta}
\def\Sg{\Sigma}
\def\vP{\varPhi}

\def\bsp{\boldsymbol{p}}
\def\bsq{\boldsymbol{q}}
\def\bsr{\boldsymbol{r}}
\def\bsmcA{\boldsymbol{\mathcal{A}}}
\def\bsmcB{\boldsymbol{\mathcal{B}}}
\def\bsmcD{\boldsymbol{\mathcal{D}}}
\def\bsmcQ{\boldsymbol{\mathcal{Q}}}

In \cite{th2}, Thurston conjectured that any open hyperbolic 3-manifold $N$ with finitely generated 
fundamental group is determined up to isometry by its end invariants.
In the case that $\fd(N)$ is a surface group, the conjecture is proved by Minsky \cite{mi2} and 
Brock, Canary and Minsky \cite{bcm}.
They also announced in \cite{bcm} that the conjecture holds for all hyperbolic 3-manifolds $N$ with $\fd(N)$ 
finitely generated.

In this paper, we concentrate on the previous case that $\fd(N)$ is isomorphic to the fundamental 
group of a compact surface $S$.
The original proof of the Ending Lamination Conjecture deeply depends on the theory of the curve 
complex developed by Masur and Minsky \cite{mm1,mm2}.
Our aim here is to replace some of such arguments (especially those concerning hierarchies) by 
arguments of standard hyperbolic geometry.

In \cite{mi2}, Minsky constructed the Lipschitz model manifold by using hierarchies in the following steps: 
(1) the definition of hierarchies, (2) the proof of the existence of a hierarchy $H_\nu$ associated to the end 
invariants $\nu$ of a given hyperbolic 3-manifold, (3) the definition of slices of $H_\nu$, 
(4) the proof of the existence of a resolution containing these slices, (5) the construction of 
the model manifold $M_\nu$ from the resolution which is realizable in $S\times \rr$.

In Section \ref{Hie}, we define a hierarchy directly as an object in $S\times \rr$, so the steps (1)-(5) as 
above are accomplished at once.
Lemma \ref{structure_sigma} is a geometric version of 
an assertion of Theorem 4.7 (Structure of Sigma) in \cite{mm2}, which plays an important role in our 
geometric proof of the bi-Lipschitz model theorem.

Section \ref{md} reviews Minsky's definition of the piecewise Riemannian metric on the model manifold.

In the proof of the Lipschitz model theorem in  \cite[Section 10]{mi2}, the hyperbolicity of 
the curve graph $\cc(S)$ is crucial.
This hyperbolicity is proved by \cite{mm1} (see also \cite{bow1}).
The proof of this theorem also needs two key lemmas. 
One of them (Lemma 7.9 in \cite{mi2}) is called the Length Upper Bounds Lemma, which shows that vertices of tight geodesics in $\cc(S)$ 
associated to the end invariants of $N$  are realized by geodesic loops in $N$ of length less than a uniform constant.  
Bowditch \cite{bow2} gives an alternative proof of this lemma by using more hyperbolic geometric techniques compared with Minsky's original proof.
Soma \cite{so} also gives a proof based on arguments in \cite{bow2}.
The proof in \cite{so} skips rather harder discussions in \cite[Sections 6 and 7]{bow2} 
by fully relying on geometric limit arguments.
The other key lemma (Lemma 10.1 in \cite{mi2}) shows that any vertical solid torus in the 
model manifold of $N$ with large meridian coefficient corresponds to a Marugulis tube in 
$N$ with sufficiently short geodesic core.
The original proof of this lemma is based on the ingenious estimations of meridian coefficients in \cite[Section 9]{mi2}.
In Section \ref{Lip}, we will give a shorter geometric proof of it. 
 
Section \ref{bLip} is the main part of this paper, where the bi-Lipschitz model theorem is proved by arguments 
of ourselves.

Alternate approaches to the Ending Lamination Conjecture are given by \cite{bow3,bbes,re}.
In \cite{bow3}, Bowditch  proved the sesqui-Lipschitz model theorem without using hierarchies.
Though the assertion of Bowditch's theorem is slightly weaker than that of the bi-Lipschitz model theorem, 
it is sufficient to prove the Ending Lamination Conjecture.
Ideas in this paper are much inspired from the philosophy of \cite{bow3}.

\section{Preliminaries}\label{Pre}

We refer to Thurston \cite{th1}, Benedetti and Petronio \cite{bp}, 
Matsuzaki and Taniguchi \cite{mt}, Marden \cite{ma} for details on hyperbolic geometry, and to 
Hempel \cite{he} for those on 3-manifold topology.
Throughout this paper, all surfaces and 3-manifolds are assumed to be oriented.

\subsection{The curve graph and tight geodesics}\label{curve}

Here we review some fundamental definitions and results on the curve graph.

Let $F$ be a connected (possibly closed) surface which has a hyperbolic metric of finite 
area such that each component of $\part F$ is a geodesic loop.
The complexity of $F$ is defined by $\xi(F)=3g+p-3$, where $g$ is the genus 
of $F$ and $p$ is the number of boundary components and punctures of $F$.

When $\xi(F)\geq 2$, we define the \emph{curve graph} $\cc(F)$ of $F$ to be the simplicial graph whose vertices are 
homotopy classes of non-contractible and non-peripheral simple closed curves in $F$ and whose edges are 
pairs of distinct vertices with disjoint representatives.
We simply call a vertex of $\cc(F)$ or any representative of the class a \emph{curve} in $F$.
For our convenience, we take a uniquely determined geodesic in $F$ as a representative for any curve in $F$.
The notion of curve graphs is introduced by Harvey \cite{har} and extended and modified versions are studied by \cite{mm1,mm2,mi1}.
In the case that $\xi(F)=1$, the curve graph $\cc(F)$ is the 1-dimensional simplicial complex  such that the vertices are 
curves in $F$ and that two curves $v,w$ form the end points of an edge if and only if they have the minimum geometric intersection number $i(v,w)$, 
that is, $i(v,w)=1$ when $F$ is a one-holed torus and $i(v,w)=2$ when $F$ is a four-holed sphere.
In either case, $\cc(F)$ is supposed to have an arcwise metric such that each edge is isometric to the unit interval $[0,1]$.
The graph $\cc(F)$ is not locally finite but is proved to be $\delta$-hyperbolic by Masur and Minsky 
\cite{mm1} (see also Bowditch \cite{bow1}) for some $\delta>0$.
The set of vertices in $\cc(F)$ is denoted by $\cc_0(F)$.
We say that the union of $k+1$ elements of $\cc_0(F)$ with mutually disjoint representatives is a \emph{$k$-simplex} in $\cc_0(F)$.

Let $\mathcal{ML}(F)$ be the space of compact measured laminations on $\Int F$ and $\mathcal{UML}(F)$ the quotient space of $\mathcal{ML}(F)$ obtained by forgetting the measures, and 
let $\mathcal{EL}(F)$ be the subspace of $\mathcal{UML}(F)$ consisting of filling laminations $\mu$.
Here $\mu$ being \emph{filling} means that, for any $\mu'\in \mathcal{UML}(F)$, either $\mu'=\mu$ or $\mu'$ 
intersects $\mu$ non-trivially and transversely.
According to Klarreich \cite{kla} (see also Hamenst\"{a}dt \cite{ham}), there exists a homeomorphism $k$ from the 
Gromov boundary $\part \cc(F)$ to $\mathcal{EL}(F)$ which is defined so that a sequence $\{v_i\}$ in $\cc_0(F)$ converges to $\beta \in \part \cc(F)$ if and only if it converges to $k(\beta)$ in $\mathcal{UML}(F)$.

\begin{definition}
A sequence $\{v_i\}_{i\in I}$ of simplices in $\cc_0(F)$ is called a \emph{tight sequence} if it satisfies 
one of the following conditions, where $I$ is a finite or infinite interval of $\zz$.
\begin{enumerate}[\rm (i)]
\item
When $\xi(F)>1$, for any vertices $w_i$ of $v_i$ and $w_j$ of $v_j$ with $i\neq j$, $d_{\cc(F)}(w_i,w_j)=|i-j|$.
Moreover, if $\{i-1,i,i+1\}\subset I$, then $v_i$ is represented by the union of components of $\part F_{i-1}^{i+1}$ 
which are non-peripheral in $F$, where $F_{i-1}^{i+1}$ is the minimum subsurface in $F$ with geodesic boundary 
and containing the geodesic representatives of all vertices of $v_{i-1}$ and $v_{i+1}$.
\item
When $\xi(F)=1$, $\{v_i\}$ is just a geodesic sequence in $\cc_0(F)$.
\end{enumerate}
\end{definition}

We regard that a single vertex is a tight sequence of length $0$.
The definition implies that, for any tight sequence $\{v_i\}$, if a vertex $w$ of $\cc(F)$ meets $v_i$ transversely, then $w$ meets at least one of $v_{i-1}$ and $v_{i+1}$ transversely.

The following theorem is Lemma 5.14 in \cite{mi2} (see also Theorem 1.2 in \cite{bow2}), which is crucial in the proof of the Ending Lamination Conjecture.

\begin{theorem}\label{tight}
Let $u,w$ be distinct points of $\cc_0(F)\cup \mathcal{EL}(F)$, there exists a tight sequence connecting 
$u$ with $w$.
\end{theorem}

Let $\boldsymbol{i},\boldsymbol{t}$ be unions of mutually disjoint curves in $F$ and laminations in $\mathcal{UML}(F)$.
Then a tight sequence $g=\{v_i\}_{i\in I}$ in $F$ is said to be a \emph{tight geodesic} with the 
\emph{initial marking} $\boldsymbol{i}(g)=\boldsymbol{i}$ and the \emph{terminal marking} $\boldsymbol{t}(g)=
\boldsymbol{t}$ if it satisfies the following conditions.
\begin{itemize}
\item
If $i_0=\inf I>-\infty$, then $v_{i_0}$ is a curve component of $\boldsymbol{i}$, otherwise $\boldsymbol{i}$ 
consists of a single lamination component and $\boldsymbol{i}=\lim_{i\sto-\infty}v_i\in \mathcal{EL}(F)$.
\item
If $j_0=\sup I<\infty$, then $v_{j_0}$ is a curve component of $\boldsymbol{t}$, otherwise $\boldsymbol{t}$ 
consists of a single lamination component and $\boldsymbol{t}=\lim_{j\sto \infty}v_j\in \mathcal{EL}(F)$.
\end{itemize}
Our rule in the definition is that, whenever an end of a tight geodesic is chosen, curve components have 
priority over lamination components if any.

\subsection{Setting on hyperbolic 3-manifolds}\label{hyp}
Throughout this paper, we suppose that $S$ is a compact connected surface (possibly $\part S=\eset$) with 
$\chi(S)<0$ and $\rho:\fd(S)\to \mathrm{PSL}_2(\mathbf{C})$ is a faithful discrete representation 
which maps any element of $\fd(S)$ represented by a component of $\part S$ to a parabolic element.
For convenience, we fix a complete hyperbolic surface $\wh S$ containing $S$ as a compact core and such that 
each component $P$ of $\wh S\setminus S$ is a parabolic cusp with $\mathrm{length}(\part P)=\ve_1$.
We denote the quotient hyperbolic 3-manifold $\hh^3/\rho(\fd(S))$ by $N_\rho$ (or $N$ for short).
By Bonahon \cite{bo}, $N$ is homeomorphic to $\wh S\times \rr$. 
Fix a 3-dimensional Margulis constant $\ve_0>0$.
For any $0<\ve<\ve_0$, the (open) $\ve$-thin and (closed) $\ve$-thick parts of $N$ are 
denoted by $N_{(0,\ve)}$ and $N_{[\ve,\infty)}$ respectively.
It is well known that there exists a constant $\ve_1>0$ depending only on $\ve$ and the topological type of $S$ 
such that, for any pleated map $f:\wh S\to N$, the image $f(\wh S(\sg_f)_{[\ve_0,\infty)})$ is 
disjoint from $N_{(0,\ve_1)}$, where $\sg_f$ is the hyperbolic structure on $\wh S$ induced from that on $N$ via $f$.
If necessary retaking $\ve_1>0$, we may assume that each simple closed geodesic in $\wh S$ is contained in $S$. 
The \emph{augmented core} $\wh C_\rho$ of $N$ is defined by 
$$\wh C_\rho=C^1_\rho\cup N_{(0,\ve_0]},$$
where $C^1_\rho$ is the closed $1$-neighborhood of the convex core of $N$ and $N_{(0,\ve]}$ is the closure of $N_{(0,\ve)}$ in $N$. 
The complement $N\setminus \Int \wh C_\rho$ is denoted by $E_N$, which is considered to be a neighborhood of the union of geometrically finite relative ends of $N$.

The orientations of $S$, $N$ and a proper homotopy equivalence $f:\wh S\to N$ with $\fd(f)=\rho$ determines 
the $(+)$ and $(-)$-side ends of $N$.
Let $\boldsymbol{q}_+=l_1\cup\cdots\cup l_n$ be the disjoint union of simple closed geodesics in $S$ corresponding to 
the parabolic cusps in the 
$(+)$-side end and let $\mathcal{GF}_+$ (resp.\ $\mathcal{SD}_+$) be the set of components of $\wh S\setminus 
\boldsymbol{q}_+$ corresponding to geometrically finite (resp.\ simply degenerate) relative ends in the $(+)$-side.
For any $F_i\in \mathcal{GF}_+$ (resp.\ $F_j\in \mathcal{SD}_+$), let $\sg_i\in \mathrm{Teich}(F_i)$ (resp.\ 
$\lambda_j\in \mathcal{EL}(F_i)$) be the conformal structure on $F_i$ at infinity (resp.\ 
the ending lamination on $F_i$), see \cite{th1,bo} for details on ending laminations.
The family $\nu_+=\{\sg_i,\lambda_j\}$ is called the $(+)$-side \emph{end invariant} set of $N$.
The $(-)$-side end invariant set $\nu_-$ is defined similarly.
The pair $\nu=(\nu_-,\nu_+)$ is the \emph{end invariant set} of $N$.

It is well known that there exists a constant $L>0$ depending only on the topological type of $S$ such that, 
for any $\sg_i\in \nu_+$ with $F_i\in \mathcal{GF}_+$, there exists a pants decomposition $\bsr_i= 
s_1\cup\cdots\cup s_{m}$ on $F_i$ such that $l_{\sg_i}(s_k)<L$, where $l_{\sg_i}(s_k)$ is the length of the 
geodesic in $F(\sg_i)$ homotopic to $s_k$.
Then the union
\begin{equation}\label{gpd}
\bsp_+=\bsq_+\cup \Bigl(\bigcup_{F_i\in \mathcal{GF}_+}\bsr_i\Bigr)\cup\Bigl(\bigcup_{F_j\in \mathcal{SD}_+}
\lambda_j\Bigr)
\end{equation}
is called a \emph{generalized pants decomposition} on $\wh S$ associated to $\nu_+$.
A \emph{generalized pants decomposition} $\bsp_-$ on $\wh S$ associated to $\nu_-$ is defined similarly.

\subsection{Annulus union and bricks}

We suppose that $\wh \rr=\{-\infty\}\cup \rr\cup \{\infty\}$ is the two-point compactification of 
$\rr$.
So $\wh\rr$ is homeomorphic to a closed interval in $\rr$.
For any subset $P$ of $\wh S\times \wh\rr$, the image of $P$ by the orthogonal projection to $\wh S$ (resp.\ $\wh\rr$) 
is denoted by $P^S$ (resp.\ $P^\rr$), that is, $P^S=\{x\in \wh S\,;\,(x,t)\in P\ \mbox{for some $t\in\wh\rr$}\}$ 
and $P^\rr=\{t\in \wh\rr\,;\,(x,t)\in P\ \mbox{for some $x\in \wh S$}\}$.
For any non-peripheral simple geodesic loop $l$ in $\wh S$ and any closed interval $J$ of $\wh\rr$, 
$A=l\times J$ is called a \emph{vertical annulus} in $S\times \wh\rr$.
For a connected open subsurface $F$ of $\wh S$ with $\fr (F)$ geodesic, the product $B=F\times J$ is called 
a \emph{brick} in $\wh S\times \wh\rr$, 
where $\fr(F)$ denotes the frontier $\ol F\cap \ol{(\wh S\setminus F)}$ of $F$ in $\wh S$.
Set $\part_{\mathrm{vt}}B=\fr(F)\times J$, $\part_-B=F\times \{\inf J\}$,  
$\part_+ B=F\times \{\sup J\}$ (possibly $\inf J=-\infty$ or $\sup J=\infty$) and $\part_{\mathrm{hz}}B=\part_-B\cup \part_+ B$.
The surface $\part_+ B$ (resp.\ $\part_-B$) is called the \emph{positive} (resp.\ \emph{negative}) \emph{front}
of $B$. 
We say that a union $\ca$ of mutually disjoint vertical annuli in $\wh S\times \wh \rr$ which are locally 
finite in $\wh S\times \rr$ is an 
\emph{annulus union}.
A \emph{horizontal surface} $F$ of $(\wh S\times \wh\rr,\ca)$ is a connected component of $\wh S\times \{a\}\setminus \ca$ for some $a\in \wh\rr$.
In particular, $\fr(F)\subset \ca$ and $F^S$ is an open subsurface of $\wh S$.
A horizontal surface $F$ is \emph{critical} with respect to $\ca$ if at least one component of 
$\fr(F)$ is an edge of some component of $\ca$.
Let $\cb$ be the set of bricks in $\wh S\times \wh \rr$ which are maximal among bricks $B$ with $\Int B\cap \ca=\eset$ and $\part_{\mathrm{vt}}B\subset \ca$, see Fig.\ \ref{f_brick}\,(a).
Note that, for any $B\in \cb$, $B\cap \ca$ is a disjoint union (possibly empty) of simple geodesic loops in $\part_{\mathrm{hz}}B$.
This fact is important in the definition of hierarchies in Section \ref{Hie}.
Each component of $\part_{\mathrm{hz}}B\setminus \ca$ is a critical horizontal surface of $(\wh S\times \wh\rr,\ca)$.
\begin{figure}[hbtp]
\centering
\scalebox{0.125}{\includegraphics[clip]{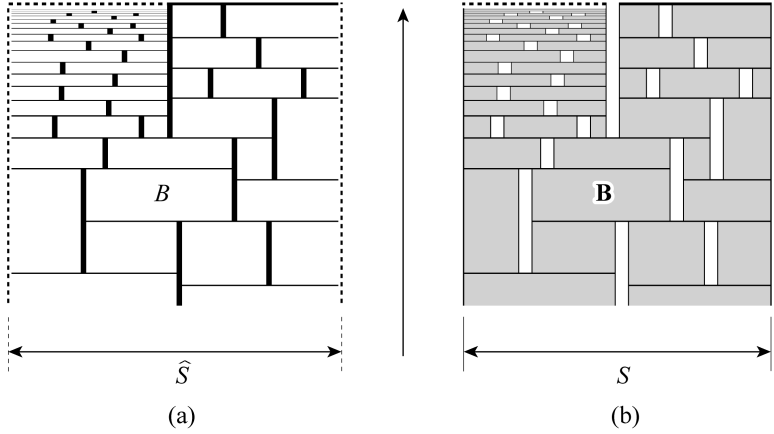}}
\caption{(a) The union of vertical segments is $\ca$.
(b) The shaded region represents $W$.}
\label{f_brick}
\end{figure}

For a vertical annulus $A=l\times J$, $U=\Int(L\times J)$ is called a \emph{vertical solid torus} 
(for short \emph{v.s.-torus}) with the \emph{geodesic core} $A$, where $L$ be an equidistant regular neighborhood of $l$ in $S$.
Then $L\times J$ is the closure $\ol U$ of $U$ in $S\times \wh \rr$.
We set $\part \ol U=\part U$ for simplicity.
A simple loop in $\part U$ is a \emph{longitude} of $U$ if it is isotopic in $\part U$ to a component of $\part A$.
A \emph{meridian} of $\part U$ is a simple loop in $\part U$ which is non-contractible in $\part U$ but contractible in $\ol U$.
For any annulus union $\ca$ in $\wh S\times \wh\rr$, there exists a disjoint union $\cv$ of v.s.-tori the union of whose geodesic cores is equal to $\ca$.
Then $\cv$ is called a \emph{v.s.-torus union} with the geodesic core $\ca$.
In general, the union $\cv^\bullet$ of the closures of components of $\cv$ is not equal to the closure $\ol \cv$ of $\cv$ in $S\times \wh \rr$.
A \emph{horizontal surface} of $(S\times \wh\rr,\cv)$ is a compact connected surface $F$ in 
$S\times \{a\}$ for some $a\in \wh\rr$ with $\Int F\cap \cv^\bullet=\eset$ and 
$\part F\subset \cv^\bullet$.
The horizontal surface is \emph{critical} if it is contained in a 
critical horizontal surface of $(\wh S\times \rr,\ca)$.
For any $B\in \cb$, the closure $\mathbf{B}$ of $B\setminus \cv^\bullet$ in $S\times \wh\rr$ is a \emph{brick} of $(S\times \wh\rr,\cv)$.  
Note that $\mathbf{B}$ is a compact subset of $S\times \wh \rr$.
The \emph{brick decomposition} $\bsmcB$ of $(S\times \wh\rr,\cv)$ is the set of bricks of $(S\times \wh\rr,\cv)$.
Then the union $W=\bigcup \bsmcB$ satisfies 
$$S\times \rr\setminus \cv\subset W \subset S\times \wh\rr\setminus \cv,$$
see Fig.\ \ref{f_brick}\,(b).
When $\mathbf{B}\in \bsmcB$ is contained in $B\in \cb$, set $\part_{\mathrm{hz}}\mathbf{B}=\part_{\mathrm{hz}} B\cap \mathbf{B}$, $\part_\pm \mathbf{B}=\part_\pm B\cap 
\mathbf{B}$ and let 
$\part_{\mathrm{vt}}\mathbf{B}$ be the closure of $\part \mathbf{B}\setminus 
\part_{\mathrm{hz}}\mathbf{B}$ in $\part\mathbf{B}$.

\subsection{Geometric limits and bounded geometry}\label{bg}

We say that a sequence $\{(N_n,x_n)\}$ of hyperbolic 3-manifolds  
with base points converges 
\emph{geometrically} to a hyperbolic 3-manifold $(N_\infty,x_\infty)$ with base point if there 
exist monotone decreasing and increasing sequences $\{K_n\}$, $\{R_n\}$ with $\lim_{n\sto\infty}K_n=1$, 
$\lim_{n\sto\infty}R_n=\infty$ and $K_n$-bi-Lipschitz maps
$$g_n:\cn_{R_n}(x_n,N_n)\to \cn_{R_n}(x_\infty,N_\infty),$$
where $\cn_R(x,N)$ denotes the closed $R$-neighborhood of $x$ in $N$.
It is well known that, if $\inf\{\mathrm{inj}_{N_n}(x_n)\}>0$, then 
$\{(N_n,x_n)\}$ has a geometrically convergent subsequence, for example see \cite{jm,bp}.
If we take a Margulis constant $\ve>0$ sufficiently small, then one can choose the bi-Lipschitz maps 
so that $g_n(\cn_{R_n}(x_n,N_n)_{[\ve,\infty)})=\cn_{R_n}(x_\infty,N_\infty)_{[\ve,\infty)}$, 
where $\cn_R(x,N)_{[\ve,\infty)}=\cn_R(x,N)\cap N_{[\ve,\infty)}$.

In general, the topological type of the limit manifold $N_\infty$ is very complicated, for example see \cite{os}.
In spite of the fact, by observing situations in geometric limits, we often know    
the existence of useful uniform constants.
We will give here typical examples.

\begin{example}\label{e_geom1}
Let $F$ be a connected compact surface and $N$ a hyperbolic 3-manifolds as in Subsection \ref{hyp}.
Suppose that $\mathrm{Teich}_{\ve}(F)$ is the Teichm\"{u}ller space such that, for any $\sg\in \mathrm{Teich}(F)$, $F(\sg)$ 
represents a hyperbolic structure on $F$ each boundary component of which is a geodesic loop of length $\ve$.
Let $f_i:F(\sigma_i)\to N_{[\ve,\infty)}$ $(i=0,1)$ be 
$K$-Lipschitz maps properly homotopic to each other in $N_{[\ve,\infty)}$, where $K\geq 1$ and 
$\sg_i\in \mathrm{Teich}_\ve(F)$ $(i=0,1)$. 
For the homotopy $H:F\times [0,1]\to N_{[\ve,\infty)}$ and a point $x\in F$, the image $H(\{x\}\times [0,1])$ 
is said to be a \emph{homotopy arc} connecting $f_0(F)$ and $f_1(F)$.
Here we will show by invoking a geometric limit argument that there exists a constant $d_0>0$ depending only 
on $\ve,d_1,K$ and the topological type of $S$ such that, if there exists a homotopy arc connecting $f_0(F)$ with 
$f_1(F)$ of length at most $d_1$, then $\dist_{\mathrm{Teich}_\ve(F)}(\sigma_0,\sigma_1)<d_0$.

Suppose contrarily that there would exist a sequence of pairs of homotopy equivalence 
$K$-Lipschitz maps $f_{i,n}:F(\sg_{i,n})\to N_{n[\ve,\infty)}$ with homotopy arcs $\alpha_n$ connecting 
$f_{0,n}(F)$ with $f_{1,n}(F)$ of length $\leq d_1$ and $\dist_{\mathrm{Teich}_\ve(F)}(\sigma_{0,n},\sigma_{1,n})\geq n$, 
where $N_n$ are hyperbolic 3-manifolds as in Subsection \ref{hyp}.
Since the $\ve/K$-thin part of $F(\sg_{i,n})$ is empty, there exists a $K'$-bi-Lipschitz map $\gamma_{i,n}:
F(\sg_0)\to F(\sg_{i,n})$ for some fixed $\sg_0\in \mathrm{Teich}_\ve(F)$, where $K'$ is a constant depending only 
on $\ve$, $K$ and $S$.
We note that $\gamma_{i,n}$ does not necessarily preserve the marking on $F$.
Let $Q_n$ be the union of bounded components of $N_{n[\ve,\infty)}\setminus f_{0,n}(F)\cup f_{1,n}(F)$ and $R_n$ 
a small regular neighborhood of $f_{0,n}(F)\cup f_{1,n}(F)$ in $N_{n[\ve,\infty)}$.
Then $J_n=R_n\cup Q_n$ is a compact connected subset of $N_{n[\ve,\infty)}$.
By \cite{fhs}, we know that $f_{0,n}$ is properly homotopic to $f_{1,n}$ in $J_n$.
If we take a base point $x_n$ of $N_n$ in $J_n$, then $\{(N_n,x_n)\}$ has a subsequence, still denoted by 
$\{N_n\}$, converges geometrically to a hyperbolic 3-manifold $(N_\infty,x_\infty)$. 
Thus we have $K_n$-bi-Lipschitz maps $g_n:\cn_{R_n}(x_n,N_n)\to \cn_{R_n}(x_\infty,N_\infty)$ as above.

For any point $y\in J_n$ with $\dist_{N_n[\ve,\infty)}(y,f_{0,n}(F)\cup f_{1,n}(F))>1$, we have a 
pleated map $g:\wh S\to N_n$ such that there exists a component $L$ of $g(\wh S)\cap N_{n[\ve,\infty)}$ 
meeting the $1$-neighborhood of $x$ in $N_{n[\ve,\infty)}$.
It is not hard to see that $L$ meets $f_{0,n}(F)\cup \alpha_n \cup f_{1,n}(F)$ non-trivially and the diameter 
of $L$ is bounded by a constant depending only on $\ve$, $S$.
Thus the diameter of $J_n$ is less than a constant $R>0$ depending only on $\ve,d_1,K,S$ and hence 
$J_n$ is contained in $\cn_{R_n}(x_n,N_n)_{[\ve,\infty)}$ for all sufficiently large $n$.

By the Ascoli-Arzel\`{a} Theorem, if necessarily passing to subsequences, one can show that 
 $\psi_{i,n}=g_n\circ f_{i,n}\circ \gamma_{i,n}:F(\sg_0)\to N_{\infty[\ve,\infty)}$ $(i=0,1)$ converge uniformly to 
$KK'$-Lipschitz maps $\varphi_i:F(\sg_0)\to N_{\infty[\ve,\infty)}$.
Since $\psi_{i,n}$ $(i=0,1)$ is properly homotopic to $\varphi_i$ for all sufficiently large $n$ and 
$f_{0,n}\circ\gamma_{0,n}$ is properly homotopic to $f_{1,n}\circ \gamma_{1,n}$ in $J_{n}$ up to marking, 
there exists a diffeomorphism (hence a $K''$-bi-Lipschitz map for some $K''\geq 1$) $\alpha:F(\sg_0)\to F(\sg_0)$ 
such that $\varphi_0$ is properly 
homotopic to $\varphi_1\circ \alpha$ in a small compact neighborhood of $g_{n}(J_{n})$ 
in $N_{\infty[\ve,\infty)}$.
This implies that, for any non-contractible simple closed curve $l$ in $F$, $\gamma_{0,n}(l)$ is homotopic to 
$\gamma_{1,n}\circ \alpha (l)$ in $F$.
Thus $\gamma_{1,n}\circ \alpha\circ \gamma_{0,n}^{-1}:F(\sg_{0,n})\to F(\sg_{1,n})$ is a marking-preserving 
${K'}^2K''$-bi-Lipschitz map for all sufficiently large $n$, which contradicts that 
$\dist_{\mathrm{Teich}_\ve(F)}(\sigma_{0,n},\sigma_{1,n})\geq n$.
This shows that the existence of our desired uniform constant $d_0$.
\end{example}

\begin{example}\label{e_geom2}
We work in the situation as in the previous example and suppose moreover that 
there exists a constant $d_2>0$ with $\dist_{N_{n[\ve,\infty)}}(f_{0,n}(F),f_{1,n}(F))\geq d_2$ for all $n$ 
and each $f_{i,n}$ is properly homotopic in $N_{n[\ve,\infty)}$ to an embedding.
By \cite{fhs}, one can suppose that such an embedding is contained in an arbitrarily small regular neighborhood 
of $f_{i,n}(F)$ in $N_{n[\ve,\infty)}$ and the image of the homotopy is in $J_n$ given as above.
Then $\varphi_i:F\to N_{\infty[\ve,\infty)}$ $(i=0,1)$ are also homotopic to embeddings $\varphi_i'$ contained in an arbitrarily small regular neighborhood of $\varphi_{i}(F)$ in $N_{\infty[\ve,\infty)}$ and the image of the homotopy is in $g_n(J_n)$ for a sufficiently large $n$.
By the standard theory of 3-manifold topology (for example see \cite{wa,he}), the union 
$\varphi_0'(F)\cup \varphi_1'(F)$ bounds a submanifold $B$ of $N_{\infty[\ve,\infty)}$ contained 
in $g_n(J_n)$ and homeomorphic to $F\times [0,1]$.
Then, for all sufficiently large $n$, $B_n=g_n^{-1}(B)$ is the submanifold of $N_{n[\ve,\infty)}$ 
such that $\fr(B_n)$ consists of two components $F_{i,n}$ $(i=0,1)$ properly homotopic to $f_{i,n}(F)$ in $J_n$.
Since the composition $g_m^{-1}\circ g_n|B_n$ defines a marking-preserving $K_mK_n$-bi-Lipschitz map from $B_n$ 
to $B_m$ and since $\lim_{m,n\sto \infty}K_mK_n=1$, we know that $B_n$'s have the geometry uniformly 
bounded by constants depending only on $\ve,d_1,d_2$ and the topological type of $S$.
\end{example}

\begin{remark}\label{r_geom}
Deform the metric on $N_{n[\ve,\infty)}$ in a small collar neighborhood of $\part N_{n[\ve,\infty)}$ so that $\part N_{n[\ve,\infty)}$ 
is locally convex but the sectional curvature of $N_{n[\ve,\infty)}$ is still pinched.
We here consider the case that $f_{i,n}:F(\sg_i)\to N_{n[\ve,\infty)}$ $(i=0,1)$ are embeddings which have the least area among 
all maps homotopic to $f_{i,n}$ without moving $f_{i,n}|_{\part F(\sg_i)}$ and such that 
$\area(F(\sg_i))$ is bounded by a constant independent of $n$.
Then the limits $\varphi_i:F\to N_{\infty[\ve,\infty)}$ are least area maps (see \cite[Lemma 3.3]{hs}), and hence 
by \cite{fhs} they are  also embeddings.
Thus, in Example \ref{e_geom2}, one can suppose that $\varphi_i'=\varphi_i$ and hence the frontier of the manifold $B$ is $\varphi_0(F)\cup \varphi_1(F)$.
\end{remark}

\section{Three-dimensional approach to hierarchies}\label{Hie}

We study hierarchies in the curve graph $\cc(S)$ introduced by \cite{mm2}.
We realize them as families 
of annulus unions  
in $\wh S\times \wh \rr$, the original idea of which is due to \cite[Section 4]{bow3}.

\subsection{Hierarchies}\label{ss_hierarchies}
Let $\bsp_\nu=(\boldsymbol{p}_-,\boldsymbol{p}_+)$ be the pair of generalized pants decompositions on $\wh S$ given 
in Subsection \ref{hyp}.
We denote by $\cb_0$ and $\wh\cb_0$ the single element set $\{\wh S\times \wh\rr\}$.
Consider a tight geodesic $g_0=\{v_i\}_{i\in I}$ with $\boldsymbol{i}(g_0)=\boldsymbol{p}_-$ and 
$\boldsymbol{t}(g_0)=\boldsymbol{p}_+$, where $I$ is an interval in $\zz$.
In this section, we always assume that, for any disjoint union $v$ of simple geodesic loop $l_1,\dots,l_k$ in $\wh S$, 
$A(v)$ represents a union of vertical annuli $A_i$ $(i=1,\dots,k)$ in $\wh S\times \wh \rr$ with $A_i^S=l_i$ and 
$A_i^\rr=A_j^\rr$ for all $i,j\in \{1,\dots,k\}$.
Thus $A(v)$ is determined uniquely from $v$ and $A(v)^\rr$.

Suppose that $\xi(S)>1$ and $\bsp_-$, $\bsp_+$ are in $\wh S\times \{-\infty\}$ and $\wh S\times \{\infty\}$ 
respectively.
When $i\in I$ is not either $\inf(I)$ or $\sup(I)$, $A(v_i)$ is defined to be the union of vertical annuli in 
$\wh S\times \rr$ with $A(v_i)^\rr=[i,i+1]$.
When $i=\sup{I}<\infty$ (resp.\ $i=\inf{I}>-\infty$), let $A(v_i)^\rr=[i,\infty]$ (resp.\ $A(v_i)^\rr=[-\infty,i+1]$).
We say that $\ca(g_0)=\bigcup_{i=0}^n A(v_i)$ is the \emph{annulus union} determined from the tight geodesic $g_0$.
Let $\cb_1$ be the brick decomposition of $(\wh S\times \wh \rr,\ca(g_0))$.
An element $B\in \cb_1$ is said to be \emph{connectable} if both $\part_\pm B\cap \ca_0$ are not empty, 
where $\ca_0=\ca(g_0)\cup \bsp_-\cup \bsp_+$.
Let $\wh\cb_1$ be the subset of $\cb_1$ consisting of connectable bricks $B$ with $\xi(B)>1$, where $\xi(B)=\xi(B^S)$.
If $\xi_{\mathrm{max}}(\cb_1)=\max \{\xi(B);B\in \cb_1\}>1$, then any $B\in \cb_1$ with 
$\xi(B)=\xi_{\mathrm{max}}(\cb_1)$ is an element of $\wh B_1$.

For any $B\in \wh\cb_1$, consider a tight geodesic $g_B$ in $B^S$ with $\boldsymbol{i}(g_B)=(\part_- B\cap \ca_0)^S$ 
and $\boldsymbol{t}(g_B)=(\part_+ B\cap \ca_0)^S$.
One can define the annulus union $\ca_B$ of vertical annuli in $B$ determined from $g_B$ as above.
In particular, $\ca_B$ consists of vertical annuli with the same width unless the length of $g_B$ is finite and 
$B^\rr\cap \{-\infty,\infty\}\neq \eset$.
Note that $\ca_B$ is a single annulus when the initial vertex of $g_B$ is equal to the terminal vertex of $g_B$.
Set $\ca_1=\ca_0\cup\bigl(\bigcup_{B\in \wh \cb_1}\ca_B\bigr)$,  $\boldsymbol{i}(\ca_B)=\part_- B\cap \ca_0$ 
and $\boldsymbol{t}(\ca_B)=\part_+ B\cap \ca_0$.

Repeating the same argument at most $\xi(S)-1$ times, say $k$ times, one can show that 
each element $B$ of the set $\cb_k$ of bricks of $(\wh S\times \wh \rr,\ca_{k-1})$ has $\xi(B)=1$.
Since $\xi_{\mathrm{max}}(\cb_k)=1$, each $B\in \cb_k$ is connectable.
We set then $\cb_k=\wh \cb_k$.
Let $g_B=\{w_i\}$ be a tight geodesic in $B^S$ with $\boldsymbol{i}(g_B)=(\part_- B\cap \ca_{k-1})^S$ and 
$\boldsymbol{t}(g_B)=(\part_+ B\cap \ca_{k-1})^S$.
Since $w_i\cap w_{i+1}\neq \eset$, we need to add a buffer brick between $A(w_i)$ and $A(w_{i+1})$ to make them mutually disjoint.
Suppose that $B^\rr=[a,b]$.
If $a\neq -\infty$ and $b\neq \infty$ and $g_B=(w_0,w_1,\dots,w_m)$, then $A(w_i)^\rr= [a+2i\tau,a+(2i+1)\tau]$ for $i=0,1,\dots,m$, where $\tau=(b-a)/(2m+1)$.
Note that $B^S\times [a+(2i+1)\tau,a+(2i+2)\tau]$ is the buffer brick between $A(w_i)$ and $A(w_{i+1})$.
If $a\neq -\infty$ and $b= \infty$ and $g_B=(w_0,w_1,\dots,w_m)$, then $A(w_i)^\rr= [a+2i,a+2i+1]$ for 
$i=0,1,\dots,m-1$ and $A(w_m)^\rr=[a+2m,\infty]$.
If $a\neq -\infty$ and $b= \infty$ and $g_B=(w_0,w_1,\dots)$, then $A(w_i)^\rr= [a+2i,a+2i+1]$ for all $i$.
In the case that $a=-\infty$, $A(w_i)$ for $w_i\in g_B$ is defined similarly.
As above, let $\ca_B=\bigcup_{w_i\in g_B} A(w_i)$,  $\boldsymbol{i}(\ca_B)=\part_- B\cap \ca_{k-1}$ and 
$\boldsymbol{t}(\ca_B)=\part_+ B\cap \ca_{k-1}$.

When $B\in \wh\cb_j$, we say that the \emph{level} of $B$ is $j$ and denote it by $\mathrm{level}(B)$.
The set $H_\nu$ of all tight geodesics appeared in this construction is called a \emph{hierarchy} 
associated to the pair $\bsp_\nu=(\bsp_-,\bsp_+)$ of generalized pants decompositions and $$\ca_{H_\nu}=\ca_{k-1}\cup
\bigl(\bigcup_{B\in \cb_{k}}\ca_B\bigr)$$
is the \emph{annulus union} determined by $H_{\nu}$.
Note that the set $H_\nu$ is not necessarily defined from $\bsp_\nu$ uniquely.

For any $B\in \wh B_j$, a maximal brick $C$ in $B$ with $\Int C\cap \ca_B=\eset$ and $\part_{\mathrm{vt}}
C\subset \ca_B$ is 
called a \emph{subbrick} of $B$.
From our construction, for any $B\in \wh\cb_j$ with $0<j\leq k$, there exists  either a brick  $B'\in \cb_{j-1}$ 
with $\part_+B'=\part_+B$ or a subbrick $C$ of some element of $\wh \cb_{j-1}$ 
with $\part_+C=\part_+B$.
In the former case, $B'$ is not in $\wh\cb_{j-1}$, otherwise $B$ would be split by $\ca_{B'}\subset \ca_{j-1}$.
Repeating  the same argument, we have eventually a brick $B_0\in \wh\cb_{j_0}$ 
for some $j_0<j$ which contains a subbrick $C$ with $\part_+C=\part_+B$.
Then we say that $B$ is \emph{directly forward subordinate} to $B_0$ and denote it by $B\ {\searrow}\!\!{}^d\ B_0$.
The \emph{directly backward subordinate} $B_0\ {}^d\!\!{\swarrow}\ B$ is defined similarly, see Fig.\ \ref{f_so}.
It is possible that $B$ is directly forward and backward to the same brick $B_0$, i.e.\ $B_0\ {}^d\!\!{\swarrow}\ B\ {\searrow}\!\!{}^d\ B_0$.
\begin{figure}[hbtp]
\centering
\scalebox{0.125}{\includegraphics[clip]{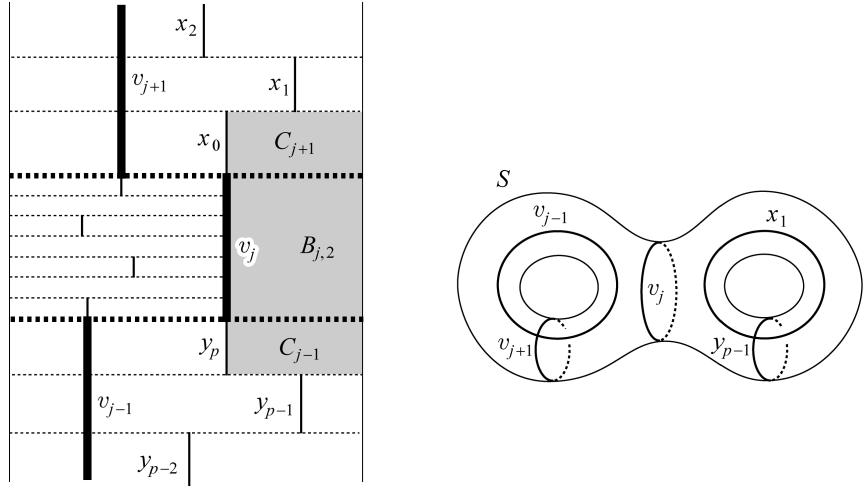}}
\caption{Let $g_0=(\dots,v_{j-1},v_j,v_{j+1},\dots)$ be a tight geodesic in the closed surface $S$ of 
genus $2$.
Let $B_a\in \wh \cb_1$ $(a=j\pm 1)$ be the element with $\part_{\mathrm{vt}}B_a=A(v_a)$.
Let $B_{j,1},B_{j,2}$ be the elements of $\cb_1$ whose vertical boundaries are $A(v_j)$ and such that 
$B_{j,1}$ is connectable but $B_{j,2}$ is not.
$(v_j=x_0,x_1,x_2,\dots)$ is a tight geodesic in $B_{j+1}^S$ and $(\dots,y_{p-1},y_{p-1},y_p=v_j)$ is a 
tight geodesic in $B_{j-1}^S$.
The shaded region represents an element $B=C_{j-1}\cup B_{j,2}\cup C_{j+1}$ of $\wh \cb_2$ with $B_{j-1}\ {}^d\!\!{\swarrow}\ B\ {\searrow}\!\!{}^d\ B_{j+1}$.
In fact, we have $\part_+ B=\part_+ C_{j+1}$ and $\part_-B=\part_- C_{j-1}$, where $C_a$ $(a=j\pm 1)$ is 
the subbrick of $B_a$ as illustrated 
in the figure.}
\label{f_so}
\end{figure}

Since only horizontal surfaces of $(\wh S\times \wh\rr,\ca_i)$ contained in $\Int B$ for some 
$B\in \wh\cb_{i+1}$ are 
split by $\ca_{i+1}$, any critical horizontal surface of $(\wh S\times \wh\rr,\ca_i)$ 
is still a (possibly non-critical) horizontal surface of $(\wh S\times \wh\rr,\ca_{i+1})$.
The relation $B\ {\searrow}\!\!{}^d\ B_0$ for $B\in \wh \cb_j$ and $B_0\in \wh \cb_{j_0}$ implies that, for any $i$ with $j_0<i\leq j$, $\part_+B$ is the positive front of some element $B_i$ of $\cb_i$.
Since $G=\part_+B\setminus \ca_{j_0}$ is a union of critical horizontal surfaces of $(\wh S\times \wh\rr,\ca_{j_0})$, each component $F$ of $G$ is a horizontal surface of $(\wh S\times \wh\rr,\ca_{j_0+1})$.
Since moreover $F\subset \part_+ B_{j_0+2}$, $F$ is critical with 
respect to $\ca_{j_0+1}$.
Repeating the same argument, one can show that $F$ is a critical horizontal surface of $(\wh S\times \wh\rr,\ca_{j-1})$.
It follows that $G=\part_+ B\setminus \ca_{j-1}=\part_+ B\setminus \boldsymbol{t}(\ca_B)$ and hence 
$\boldsymbol{t}(\ca_B)=\part_+B\cap \ca_{j_0}$.

\subsection{Single brick occupation}\label{long_bricks}

Let $\ca_0,\dots,\ca_{k-1},\ca_{H_\nu}$ be the annulus unions and $\cb_0,\dots,\cb_k$ the 
brick decompositions given in Subsection \ref{ss_hierarchies}.

\begin{lemma}\label{non_para}
Any two components of $\ca_{H_\nu}$ are not parallel in $S\times \rr$.
\end{lemma}
\begin{proof}
Suppose that $\ca_{H_\nu}$ contains distinct mutually parallel components $A,A'$.
When more than one elements are parallel to $A$, we may assume that $A'$ is closest to $A$ among them 
and $\max A^\rr<\min {A'}^\rr$.
Let $B$ (resp.\ $B'$) be the element of $\wh \cb_k$ with $\part_+ A\subset \mathrm{Int}B$ (resp.\ 
$\part_- A'\subset \mathrm{Int}B'$).
Since any two components of $\ca_B$ are not mutually parallel, $\Int B\cap \Int B'$ is empty.
Consider a pair of two directly subordinate sequences
\begin{equation}\label{e_so}
B_0\ {\searrow}\!\!{}^d\ B_1\ {\searrow}\!\!{}^d\ \cdots\ {\searrow}\!\!{}^d\ B_{m+1},\quad
B_{n+1}'\ {}^d\!\!{\swarrow}\ \cdots\ {}^d\!\!{\swarrow}\ B_1'\ {}^d\!\!{\swarrow}\ B'_0
\end{equation}
satisfying the following conditions.
\begin{enumerate}[\rm (i)]
\item
$B_0=B$, $B_0'=B'$, and $\Int B_i\cap \Int B_j'=\eset$ for any $0\leq i\leq m$ and $0\leq j\leq n$.
\item
The pair (\ref{e_so}) has the minimum $\max\{\mathrm{level}(B_{m+1}),\mathrm{level}(B_{n+1}')\}$ 
among all pairs of subordinate sequences satisfying the condition (i).
\end{enumerate}
Note that any $B_i$ and $B_j'$ meet the vertical annulus $A_0$ with $\part_-A_0=\part_- A$ and 
$\part_+ A_0=\part_+A'$ non-trivially.

First, we will show that $B_{m+1}=B_{n+1}'$.
For the symmetricity, we may assume that $\mathrm{level}(B_{m+1})\leq \mathrm{level}(B_{n+1}')$.
Take the entry $B_i$ in the the directly forward subordinate sequence with
$$\mathrm{level}(B_{i+1})\leq \mathrm{level}(B_{n+1}')< \mathrm{level}(B_{i}).$$
Then there exists an element of $D\in \cb_a$ with $D\setminus \part_- D\supset \part_+ B_i$, where 
$a=\mathrm{level}(B_{n+1}')$.
Then, in particular, $(\part_+B_i)^{\rr}\leq (\part_+D)^{\rr}$.
Suppose that $D\neq B_{n+1}'$.
Since $A$ penetrates both $D$ and $B_{n+1}'$, this implies $(\part_+D)^{\rr}\leq (\part_- B_{n+1}')^{\rr}$.
If $D\in \wh \cb_a$, then $B_i\ {\searrow}\!\!{}^d\ D$ and hence $D=B_{i+1}$.
Since then $\Int B_{i+1}\cap \Int B_{n+1}'=\eset$,
$$
B_0\ {\searrow}\!\!{}^d\ B_1\ {\searrow}\!\!{}^d\ \cdots\ {\searrow}\!\!{}^d\ B_{i+1}\ {\searrow}\!\!{}^d\ B_{i+2},\quad
B_{n+2}'\ {}^d\!\!{\swarrow}\ \cdots\ {}^d\!\!{\swarrow}\ B_1'\ {}^d\!\!{\swarrow}\ B'_0
$$
is a sequence satisfying the condition (i) and $\max\{\mathrm{level}(B_{i+2}), \mathrm{level}(B_{n+2}')\}<a$.
If $D\in \cb_a\setminus \wh \cb_a$, then $D\neq B_{i+1}$ and hence $\mathrm{level}(B_{i+1})<a$.
Thus
$$
B_0\ {\searrow}\!\!{}^d\ B_1\ {\searrow}\!\!{}^d\ \cdots\ {\searrow}\!\!{}^d\ B_{i+1},\quad
B_{n+2}'\ {}^d\!\!{\swarrow}\ \cdots\ {}^d\!\!{\swarrow}\ B_1'\ {}^d\!\!{\swarrow}\ B'_0
$$
is a sequence satisfying the condition (i) and $\max\{\mathrm{level}(B_{i+1}), \mathrm{level}(B_{n+2}')\}<a$.
In either case, this contradicts the minimality condition (ii).
It follows that $D=B_{n+1}'$.
Since this implies $D\in \wh \cb_a$, $B_{i+1}=D$.
Thus we have $i=m$ and $B_{m+1}=B_{n+1}'=D$.

For short, set $D^S=F$, $A_0^S=l$, $v=\fr(\part_+B_m)$, $w=\fr(\part_- B_n')$ and let $t_m$ be the 
component of $\boldsymbol{t}(\ca_{B_m})$ such that $t_m^S$ is the terminal vertex of $g_{B_m}$.
Since $A_0\cap (\fr(\part_+ B_m)\cup \fr(\part_- B_n'))=\eset$, 
$$d_{\cc(F)}(v^S,w^S)\leq d_{\cc(F)}(v^S,l)+d_{\cc(F)}(l,w^S)= 2.$$
Suppose first that $d_{\cc(F)}(v^S,w^S)=2$ and consider the union $J$ of components of $\ca(g_D)$ with 
$(\part_- J)^\rr=v^\rr$ and $(\part_+ J)^\rr=w^\rr$, see Fig.\ \ref{f_np}.
\begin{figure}[hbtp]
\centering
\scalebox{0.125}{\includegraphics[clip]{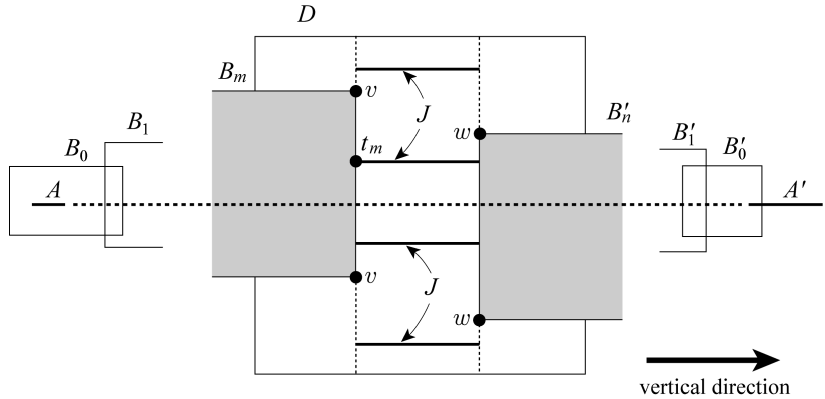}}
\caption{The case of $d_{\cc(F)}(v^S,w^S)=2$.}
\label{f_np}
\end{figure}
Since $l\cap (v^S\cap w^S)=\eset$, the tightness of $g_D$ implies either $l\subset J^S$ or $l\cap J^S=\eset$.
However, the former does not occur since $A$ and $A'$ are a closest pair.
So, we have $A_0\cap t_m=\eset$.
When $d_{\cc(F)}(v^S,w^S)=1$, either $t_m\cap \part_- B_n'=\eset$ or $t_m\subset w$ holds.
This also implies $A_0 \cap t_m=\eset$.

Repeating the same argument for $B_{m-1},B_{m-2},\dots,B_0=B$, one can show that $A_0\cap t_0=\eset$.
This contradicts that the surface $\part_+ B$ with $\xi(\part_+ B)=1$ can not contain mutually disjoint two curves.
Thus any two components of $\ca_{H_\nu}$ are not parallel to each other.
\end{proof}

The following lemma is a geometric version of the fourth assertion of Theorem 4.7 (Structure of Sigma)
in \cite{mm2}.

\begin{lemma}\label{structure_sigma}
Suppose that $B,B'$ are elements of $\wh\cb_a$ and $\wh\cb_b$ respectively.
If $B^S=B'{}^S$, then $B=B'$.
\end{lemma}
\begin{proof}
We suppose that $B\neq B'$ and induce a contradiction.

Since any two elements of $\wh B_a$ have mutually disjoint interiors, 
if $\Int B\cap \Int B'\neq\eset$, then $a\neq b$, say $a<b$.
The assumption $B\in \wh\cb_a$ implies $\ca_B\subset \ca_a\subset \ca_{b-1}$.
Since $B^S=B'{}^S$, $\Int B\cap \Int B'\neq\eset$ implies $\Int B'\cap \ca_B\neq \eset$.
This contradicts the fact that $\Int B'\cap \ca_{b-1}(\supset \Int B'\cap \ca_B)$ is empty.
Thus we have $\Int B\cap \Int B'=\eset$.

Now, we consider a sequence 
$$
B=B_0\ {\searrow}\!\!{}^d\ B_1\ {\searrow}\!\!{}^d\ \cdots\ {\searrow}\!\!{}^d\ B_{m+1}
=D=
B_{n+1}'\ {}^d\!\!{\swarrow}\ \cdots\ {}^d\!\!{\swarrow}\ B_1'\ {}^d\!\!{\swarrow}\ B'_0=B'
$$
as in the proof of Lemma \ref{non_para}.
Let $E$ be the brick in $\wh S\times \wh\rr$ with $\part_- E=\part_- B$ and 
$\part_+ E=\part_+ B'$.
We set $E^S=H$ and $H_j=E\cap \part_+B_j$.
Since $H\subset \part_+B_m^S\cap \part_-B_{n-1}^S$, 
the smallest surface $F'$ in $F=D^S$ with geodesic boundary and containing $F\setminus \Int(\part_+B_m^S\cap \part_-B_{n-1}^S)$ is disjoint from $\Int H$.
Since $g_D$ is a tight geodesic in $F$, the terminal vertex $t_m^S$ 
of $g_{B_m}$ with $t_m\subset \boldsymbol{t}(\ca_{B_m})$ is contained in $F'$ 
and hence $\Int H_m\cap t_m=\eset$.
Repeating the same argument for $B_{m-1},\dots,B_0=B$, one can show that 
$\Int (\part_+ B_0)=\Int H_0$ is disjoint from $t_0$.
This contradicts that $B_0$ is a connectable brick with $B_0\ {\searrow}\!\!{}^d\ B_1$.
Thus we have $B=B'$.
\end{proof}

Let $B$ be an element of $\cb_i$.
If $B$ is not connectable, then  $\Int B\cap \ca_i=\eset$.
Thus there exists a $C\in \cb_{i+1}$ with $C\supset B$ and $C^S=B^S$ (possibly $B=C$).
Repeating the same argument if $C$ is not connectable, we have eventually a unique element 
$B^\vee$ of $\wh\cb_j$ with $j\geq i$, $B^\vee\supset B$ and $B^{\vee S}=B^S$, which is called the 
\emph{expanding connectable brick} of $B$.
For example, $C_{j-1}\cup B_{j,2}\cup C_{j+1}\in \wh\cb_2$ in Fig.\ \ref{f_so} is the expanding connectable 
brick of $B_{j,2}\in \cb_1$.

The following lemma suggests that a large part of any longer brick $Q$ in $\wh S\times \wh\rr$ 
with $\part_{\mathrm{vt}}Q\subset \ca_{H_\nu}$ 
is occupied by a single brick in $\wh\cb_a$ for some $a$.

\begin{lemma}[Single brick occupation]\label{l_long_bricks}
There exists an integer $n_0$ depending only on $\xi(S)$ such that, for any brick $Q$ in $\wh S\times 
\wh\rr$ with $\xi(Q)\geq 1$ and $\part_{\mathrm{vt}}Q\subset \ca_{H_\nu}$, there is a set $\cb_Q=\{B_1,\dots,B_n\}$ of 
bricks in $Q$ with
$\part_{\mathrm{vt}}B_i\subset \ca_{H_\nu}$ and 
satisfying the following conditions.
\begin{enumerate}[\rm (i)]
\item
$n\leq n_0$ and $\ol{\bigcup \cb_Q}=\ol B_1\cup\cdots\cup \ol B_n\supset Q$.
\item
For at most one of the elements of  $\cb_Q$, say $B_1$, there exists a brick $C$ in $\wh\cb_a$ with 
$C^S=Q^S$ and $C\cap Q=B_1$ for some $a$.
For all other bricks $B_i$ of $\cb_Q$, $\part_{\mathrm{vt}}B_i\cap \Int Q$ is non-empty. 
\end{enumerate}
\end{lemma}

We note that $B_i$ are not necessarily elements of $\cb_a$ $(a=0,\dots,k)$.

\begin{proof}
When $Q^S=\wh S$, the pair $C=\wh S\times \wh \rr\in \wh \cb_0$ and $\cb_Q=\{Q\}$ satisfy the 
conditions (i) and (ii).
So we may assume that $Q^S\neq \wh S$ or equivalently $\part_{\mathrm{vt}}Q\neq \eset$.
In particular, $\xi(\wh S)>1$.
Recall that, for each entry $v_i$ of the tight geodesic $g_0=\{v_i\}$ in $\wh S$, $A(v_i)$ is contained in $\ca_{H_\nu}$.
Since $\part_{\mathrm{vt}}Q\subset \ca_{H_\nu}$, $A(v_i)^{\rr}\cap \Int Q^{\rr}\neq \eset$ means that 
$d_{\cc(\wh S)}(w_i,x)\leq 1$ for any vertices $w_i$ of $v_i$ and any component $x$ of 
$\part_{\mathrm{vt}}Q^S$.
It follows that $A(v_i)^{\rr}\cap \Int Q^{\rr}\neq \eset$ for at most three succeeding entries $v_i$ of $g_0$.
Thus the brick decomposition of $(Q,\ca_0\cap Q)$ consists of at most $-3\chi(Q^S)$ subbricks 
$C_1,\dots,C_m$ of $Q$.
Let $\cb_Q^{(0)}$ be the set of $C_i$ with $\part_{\mathrm{vt}} C_i\cap \Int Q\neq \eset$.
For any $C_i$ not in $\cb_Q^{(0)}$, there exists a unique $D_i$ of $\cb_1$ with $D_i\cap Q\supset C_i$.
Let $\cb_Q^{(1)}$ be the set of $C_i$ with $D_i^S=Q^S$.

Suppose that $C_i$ is not in $\cb_Q^{(0)}\cup \cb_Q^{(1)}$.
Then $Q^S$ is a proper subsurface of $D_i^S$ and $\Int D_i\cap \part_{\mathrm{vt}}Q$ is not empty.
We repeat the argument as above for $(D_i^\vee,D_i^\vee\cap Q)$ instead of $(\wh S\times \wh\rr,Q)$, 
where $D_i^\vee$ is the expanding connectable brick of $D_i$.
Then we have the sets $\cb_{D_i^\vee\cap Q}^{(0)}$ and $\cb_{D_i^\vee\cap Q}^{(1)}$ of bricks in $D_i^\vee\cap Q$ 
as above.
Since $1<\xi(D_i^\vee)<\xi(\wh S)$, this repetition finishes at most $\xi(\wh S)-\xi(Q)$ times.
Eventually we have at most $(-3\chi(Q^S))^{\xi(\wh S)-\xi(Q)}$ bricks $B_j'$ in $Q$ with 
$\bigcup_j \ol{B_j'}\supset Q$, $\part_{\mathrm{vt}}B_j'\subset \ca_{H_\nu}$ such that either $\part_{\mathrm{vt}}B_j'\cap \Int Q\neq \eset$ or 
there exists an element $D_j^\vee\in \wh\cb_a$ for some $a$ with $D_j^\vee\supset B_j'$ and 
$D_j^{\vee S}=Q^S$.
By Lemma \ref{structure_sigma}, all $D_j^\vee$ appeared in the latter case are the same brick $C$.
The set $\cb_Q$ consisting of all $B_j'$ in the former case and $Q\cap C$ (if the latter case occurs) 
satisfies 
the conditions (i) and (ii) by setting $n_0=(-3\chi(\wh S))^{\xi(\wh S)-1}$.
\end{proof}

\section{The model manifold}\label{md}

We will define the model manifold and a piecewise Riemannian metric on it as in \cite[Section 8]{mi2}.

A constant $c$ is said to be \emph{uniform} if $c$ depends only on the topological type of $S$ and previously 
determined uniform constants, and independent of the end invariants $\nu=(\nu_-,\nu_+)$.
Throughout the remainder of this paper, for a given constant $k$, a uniform constant $c(k)$ means that it 
depends only on previously determined uniform constants and $k$.

\subsection{Metric on the brick union}\label{M_on_bu}
Let $\ca=\ca_{H_\nu}$ be the annulus union associated to $H_\nu$ given in Section \ref{Hie} 
and $\cv$ a v.s.-torus union with the geodesic core $\ca$.
Let $\bsmcB$ be the brick decomposition of $(S\times \wh \rr,\cv)$ and let $W=\bigcup\bsmcB$.
Recall that for any $B\in\bsmcB$, $\xi(B)=\xi (B^S)$ is either zero or one.
Suppose that $\Sg_{0,3}$ is a hyperbolic three-holed sphere such that each component of $\part \Sg_{0,3}$ is a 
geodesic loop of length $\ve_1$, where $\ve_1$ is the constant given in Subsection \ref{hyp}.
Let $B_{0,3}$ be the product metric space $\Sg_{0,3}\times [0,1]$.
Let $\Sg_{0,4}$ be a four-holed sphere which has two essential simple closed curves $l_0,l_1$ with 
the geometric intersection number $i(l_0,l_1)=2$, and let $B_{0,4}=\Sg_{0,4}\times [0,1]$ topologically.
Let $A_i$ $(i=0,1)$ be a regular neighborhood of $l_i \times \{i\}$ in $\Sg_{0,4}\times \{i\}$.
Suppose that $B_{0,4}$ has a piecewise Riemannian metric such that each component of $\Sg_{0,4}\times \{i\}
\setminus  \Int A_i$ is isometric to the hyperbolic surface $\Sg_{0,3}$, each component of $A_0\cup A_1\cup \part_{\mathrm{vt}}B$ 
is isometric to the product annulus $S^1(\ve_1)\times [0,1]$ and 
$\dist_{B_{0,4}}(\part_-B_{0,4},\part_+ B_{0,4})=1$, where $S^1(\ve_1)$ is a round circle in the Euclidean plane of 
circumference $\ve_1$.
Let $\Sg_{1,1}$ be a fixed one-holed torus $\Sg_{1,1}$ with geodesic boundary of length $\ve_1$ and 
essential simple closed curves $l_0,l_1$ with $i(l_0,l_1)=1$.
Then a piecewise Riemannian metric on $B_{1,1}=\Sg_{1,1}\times [0,1]$ is defined similarly.
We note that these metrics are independent of $\nu$.

For any element $B\in \bsmcB$ of type $(i,j)\in \{(0,3),(0,4),(1,1)\}$, consider a diffeomorphism  
$h_B:B_{i,j}\to B$ such that $h_B(\part_{\mathrm{vt}}B_{i,j})=\part_{\mathrm{vt}}B$ and 
moreover $h_B(A_\pm)=\part_\pm B\cap \cu$ when $\xi(B)=1$, where $A_-=A_0$ and $A_+=A_1$.
One can choose these homeomorphisms so that, for any $B,B'$ in $\bsmcB$ with 
$F=\part_+ B\cap \part_- B'\neq \eset$, $(h_B|_{h_B^{-1}(F)})\circ (h_{B'}|_F)^{-1}$ is 
an isometry.
Then $W$ has the piecewise Riemannian metric induced from those on $B_{0,3},B_{0,4},B_{1,1}$ via 
embeddings $h_B:B\to W$.
Since any automorphism $\eta:\Sg_{0,3}\to \Sg_{0,3}$ is isotopic to a unique isometry, the metric on $W$ is 
uniquely determined up to ambient isotopy.

\subsection{Construction of the model manifold}\label{C_mm}

We extend $W$ to the manifold $M_\nu [0]$ with piecewise Riemannian metric as in \cite[Subsections 3.4 and 8.3]{mi2}.
For any subset $C$ of $S$, we set $C\times \{\infty\}=C^{\{+\}}$ and $C\times \{-\infty\}=C^{\{-\}}$.

Let $\cv_{\mathrm{p.c.}}$ (resp.\ $\cv_{\mathrm{g.f.}}$) be the union of components $U$ of $\cv$ such that the closure $\ol U$ in $S\times \wh\rr$ contains a component of $\bsq_-^{\{-\}}\cup \bsq_+^{\{+\}}$ 
(resp.\ $\bsr_-^{\{-\}}\cup \bsr_+^{\{+\}}$), where $\bsr_\pm=\bigcup_{F_i\in\mathcal{GF}_\pm}\bsr_i$.
If we denote the complement $\cv\setminus (\cv_{\mathrm{p.c.}}\cup \cv_{\mathrm{g.f.}})$ by 
$\cv_{\mathrm{int.}}$, then $\cv$ is represented by the disjoint union
$$
\cv=\cv_{\mathrm{int.}}\cup \cv_{\mathrm{g.f.}}\cup \cv_{\mathrm{p.c.}}.
$$
For any $F_i$ in $\mathcal{GF}_+$ (resp.\ in $\mathcal{GF}_-$), we suppose that $F_i=F_i^{\{+\}}$ (resp.\ 
$F_i=F_i^{\{-\}}$) and denote the closure of $(F_i\cap S^{\{\pm\}})\setminus \ol\cv_{\mathrm{p.c.}}$ in $S^{\{\pm\}}$ by $F_i'$, see Fig.\ \ref{f_model}\,(a).
Thus $F_i'$ is a compact surface obtained from $F_i$ by deleting the parabolic cusp components. 
For the conformal structure $\sg_i\in \mathrm{Teich}(F_i)$ at infinity given in 
Subsection \ref{hyp}, consider the conformal rescaling $\tau_i$ of 
$\sg_i\in \mathrm{Teich}(F_i)$ such that $\tau_i/\sg_i$ is a continuous map which is equal to $1$ on 
$F_i(\sg_i)_{[\ve_1,\infty)}$ and each component of $F_i(\sg_i)_{(0,\ve_1]}$ is a Euclidean cylinder with 
respect to the $\tau_i$-metric.
There exists a piecewise Riemannian metric $\upsilon_i$ on $F_i'$ such that 
$F_i'(\upsilon_i)_{(0,\ve_1]}$ is equal to $F_i'\cap \ol\cu_{\mathrm{g.f.}}$, 
each component of $F_i'(\upsilon_i)_{(0,\ve_1]}$ is isometric to a Euclidean cylinder $S^1(\ve_1)\times [0,n]$ 
with $n\in \nn$, 
and each component of $F_i'(\upsilon_i)\setminus \Int F_i'(\upsilon_i)_{(0,\ve_1]}$ is isometric to $\Sg_{0,3}$. 
It is not hard to choose such a metric $\upsilon_i$ so that the identity $F_i'(\tau_i)
\to F_i'(\upsilon_i)$ is uniformly bi-Lipschitz.
Note that our $\upsilon_i$ corresponds to the metric ${\sg^m}'$ given in \cite[Subsection 8.3]{mi2}.
Endow the union $R_i=F_i'\times [-1,0]\cup \part F_i'\times [0,\infty)$ with a piecewise Riemannian metric such 
that (i) $F_i'\times \{-1\}$ is equal to $F_i'(\upsilon_i)$, (ii) $F_i'\times \{0\}\cup 
\part F_i'\times [0,\infty)$ is isometric $F_i(\tau_i)$ via an isometry whose restriction on $F_i'$ is the 
identity, (iii) $\part F_i'\times [-1,0]$ is a Euclidean cylinder of width 1 and (iv) the identity from 
$F_i'\times [-1,0]$ to the product metric space $F_i'(\upsilon_i)\times [-1,0]$ is 
uniformly bi-Lipschitz.
We call that the metric space $R_i$ is a \emph{boundary brick} associated to $\sg_i\in \mathrm{Teich}(F_i)$ 
for $F_i\in\mathcal{GF}_+$.
A \emph{boundary brick} associated to $\sg_j\in \mathrm{Teich}(F_j)$ 
for $F_j\in\mathcal{GF}_-$ is defined similarly.
Then $M_\nu[0]$ is the metric space obtained by attaching $R_i$ to $W$ for any $F_i\in \mathcal{GF}_a$ $(a=\pm)$ 
by the isometry 
$(\part_a B_1\cup\cdots\cup \part_a B_m)\times \{-1\}\to \part_a B_1\cup\cdots\cup \part_a B_m$ 
isotopic to the identity, where $B_1,\dots,B_m$ are the elements of $\bsmcB$ meeting $F_i'$ 
non-trivially, see Fig.\ \ref{f_model}\,(b).
\begin{figure}[hbtp]
\centering
\scalebox{0.125}{\includegraphics[clip]{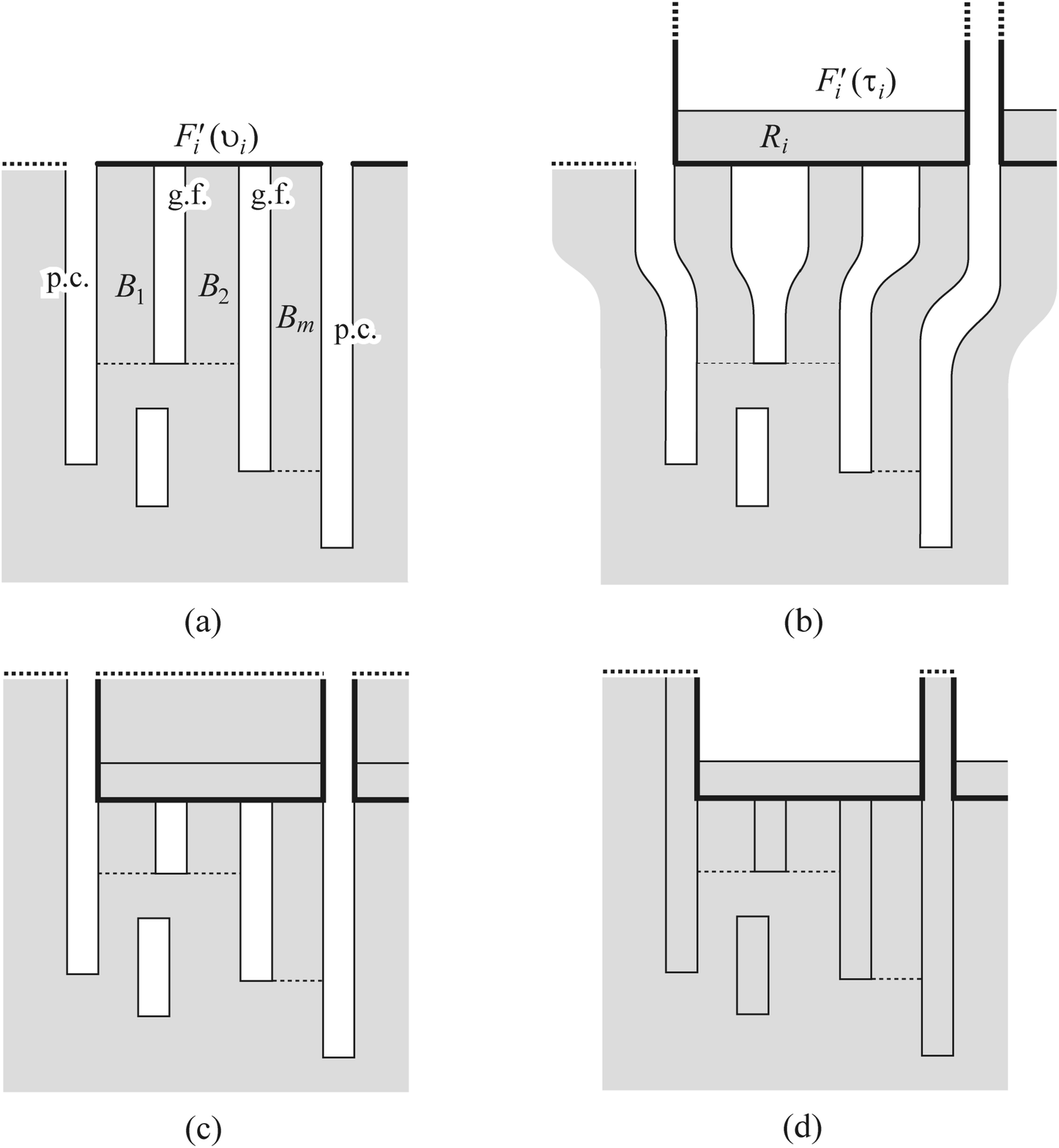}}
\caption{(a) Each white rectangle labeled with `p.c.' (resp.\ `g.f.') represents a component of $\cv_{\mathrm{p.c.}}$ (resp.\ $\cv_{\mathrm{g.f.}}$).
The shaded regions in (a)--(d)  represent $W$, $M_\nu [0]$, $M\!E_\nu[0]$ and $M_\nu$ respectively.}
\label{f_model}
\end{figure}

Extend furthermore $M_\nu[0]$ by attaching the spaces $F_i\times [0,\infty)$ with metric 
$ds^2=\tau_i e^{2r}+dr^2$ $(r\in [0,\infty))$ for $F_i\in \mathcal{GF}_a$ $(a=\pm)$ to $M_\nu[0]$ by identifying 
$F_i\times \{0\}$ with the `outer boundary' $F_i'\times \{0\}\cup \part F_i'\times [0,\infty)$ of $R_i$.
We set the extended manifold $M_\nu[0]\cup E_\nu$ by $M\!E_\nu[0]$, where $E_\nu=\bigcup_{F_i\in \mathcal{GF}_+\cup \mathcal{GF}_-} F_i\times [0,\infty)$.
From our construction, we can re-embed $M\!E_\nu[0]$ to $S\times \rr$ so that there exists a homeomorphism 
$\eta:\cv\to \wh S\times \rr\setminus M\!E_\nu[0]\subset \wh S\times \rr$ isotopic to the inclusion 
$\cv\subset \wh S\times \rr$ 
and such that, for any 
component $U$ of $\cv\setminus \cv_{\mathrm{g.f.}}$, $\eta|_U$ is the identity, see Fig.\ \ref{f_model}\,(c).
We denote   
$\eta(\cv_{\mathrm{int.}})$ by $\cu_{\mathrm{int.}}$, $\eta(\cv_{\mathrm{g.f.}})$ by $\cu_{\mathrm{g.f.}}$ 
and $\eta(\cv_{\mathrm{p.c.}})\cup \cu_{(\wh S\setminus S)}$ by $\cu_{\mathrm{p.c.}}$ respectively, 
where $\cu_{(\wh S\setminus S)}=(\wh S\sm S)\times \rr$. 
Then the complement $\cu=\wh S\times \rr\setminus M\!E_\nu[0]$ is represented by the disjoint union
\begin{equation}\label{cu}
\cu=\cu_{\mathrm{int.}}\cup \cu_{\mathrm{g.f.}}\cup \cu_{\mathrm{p.c.}}.
\end{equation}
For any component $U$ of $\cu$, the frontier $\part U$ of $U$ in $\wh S\times \rr$ 
is a torus if $U\subset \cu\setminus \cu_{\mathrm{p.c.}}$, otherwise $\part U$ is an open annulus.
We set here
$$M_\nu=M_\nu[0]\cup \cu\quad\mbox{and}\quad M\!E_\nu=M_\nu\cup E_\nu\ (=\wh S\times \rr).$$

\subsection{Meridian coefficients}\label{ss_meridian}
Let $U=U(v)$ denote the component of $\cu\setminus \cu_{(\wh S\setminus S)}$ such that $\eta^{-1}(U)\subset \cv$ is a v.s.-torus with geodesic core $A(v)$.
From our construction of the metric on $M_\nu [0]$, any component $\part U(v)$ is a 
Euclidean cylinder which has the foliation $\cf_U=\cf_v$ consisting of geodesic longitudes of length $\ve_1$.
For any complex number $z$ with $\mathrm{Im}(z)>0$ and $\eta>0$, we denote the quotient map 
$\mathbf{C}\to \mathbf{C}/\eta(\zz+z\zz)$ by $\pi_{z,\eta}$. 
If $U\subset \cu\setminus \cu_{\mathrm{p.c.}}$, then we have a unique $\omega\in \mathbf{C}$ with 
$\mathrm{Im}(\omega)>0$ such that there exists an 
orientation-preserving isometry from the quotient space $\mathbf{C}/\ve_1(\zz+\omega\zz)$ to $\part U$ which 
maps $\pi_{w,\ve_1}(\rr)$ (resp.\ $\pi_{w,\ve_1}(\omega\rr)$) to a longitude (resp.\ a meridian) of $U$.
We denote the $\omega$ by $\omega_M(U)$ or $\omega_M(v)$ and call it the \emph{meridian coefficient} of $\part U$.
If $U\subset \cu_{\mathrm{p.c.}}$, then we define $\omega_M(U)=\sqrt{-1}\infty$.
Note that $\ve_1\mathrm{Im}(\omega_M(U))$ is a positive integer whenever $U\subset \cu\setminus \cu_{\mathrm{p.c.}}$.
In fact, the brick decomposition $\bsmcB$ induces the decomposition on $\part U$ consisting of two 
horizontal annuli with integer width and $\ve_1\mathrm{Im}(\omega_M(U))-2$ vertical annuli of width one.

For any integer $k>0$, consider the union $\cu[k]$ of components $U$ of $\cu$ with $|\omega_M(U)|\geq k$ and 
$$M_\nu [k]=M_\nu[0]\cup (\cu\setminus \cu[k])\quad\mbox{and}\quad M\!E_\nu[k]=M_\nu[k]\cup E_\nu.$$
Thus $M_\nu=M_\nu[k] \cup \cu[k]$ and $M\!E_\nu=M\!E_\nu[k]\cup \cu[k]$.
We suppose that each component $U$ of $\cu\setminus \cu[k]$ has a Riemannian metric  
extending the Euclidean metric on $\part U$ and isometric to a hyperbolic tube with geodesic core.
These metrics define piecewise Riemannian metrics on $M_\nu[k]$ and $M\!E_\nu[k]$.

\section{The Lipschitz model theorem}\label{Lip}

The Lipschitz Model Theorem given in \cite{mi2} is a homotopy equivalence map 
from $M_\nu$ to the augmented core $\wh C_\rho$ of $N_\rho$ such that the restriction to $M_\nu[k]$ is a $K$-Lipschitz map 
for some uniform constant $K$ independent of $\nu,\rho$.
The following is the precise statement.

\begin{theorem}[Lipschitz Model Theorem]\label{lm}
There exists a degree-one, homotopy equivalence map $f:M_\nu\to \wh C_\rho$ with $\fd (f)=\rho$ and satisfying 
the following conditions, where $K\geq 1,k\in \nn$ are constants independent of $\nu,\rho$.
\begin{enumerate}[\rm (i)]
\item
The image $\bbt[k]=f(\cu[k])$ is a union of components of $N_{\rho(0,\ve_1)}$ with 
$\bbt[k]\supset N_{\rho(0,\ve_2)}$ for some uniform constant $0<\ve_2\leq \ve_1$ and 
the restriction $f|_{\cu[k]}:\cu[k]\to \bbt[k]$ defines a bijection between the components of $\cu[k]$ and 
$\bbt[k]$.
\item
$f(M_\nu[k])=\wh C_\rho[k]$ and the restriction $f|_{M_\nu[k]}:M_\nu[k]\to \wh C_{\rho}[k]$ is 
a $K$-Lipschitz map, where $\wh C_\rho[k]=\wh C_\rho\setminus \bbt[k]$.
\item
The restriction $f|_{\part M_\nu}:\part M_\nu\to \part\wh C_\rho$ is a $K$-bi-Lipschitz homeomorphism 
which can be extended to a $K$-bi-Lipschitz map $f':E_\nu\to E_N$ and moreover to a conformal 
map from $\part_\infty M\!E_\nu$ to $\part_\infty N_\rho$.
(Moreover, one can construct the map $f$ so that, for any boundary brick $R_i$, $f|_{R_i}:R_i\to f(R_i)$ is $K$-bi-Lipschitz 
and $f^{-1}(f(R_i))=R_i$.)
\end{enumerate}
\end{theorem}

The proof starts with the restriction $f_0:M_\nu\to N_\rho$ of a marking-preserving homeomorphism 
$S\times \rr\to N_\rho$.
Minsky's proof needs the following two lemmas which correspond to Lemmas 7.9 and 10.1 in \cite{mi2} respectively.

\begin{lemma}[Length Upper Bounds]\label{ub}
There exists a uniform constant $d_0$ such that, for any vertex $v$ appeared in $H_\nu$, $l_\rho(v)\leq d_0$.
\end{lemma}

Recall that $H_\nu$ is the hierarchy defined in Section \ref{Hie}.
For any curve $c$ in $M_\nu$, $l_\rho(c)$ denotes the length of the geodesic in 
$N_\rho$ freely homotopic to $f_0(c)$ if any and otherwise $l_\rho(c)=0$.
We also define $l_\rho(v)=l_\rho(c)$ for a curve $v$ in $S$ with $v=c^S$. 
As was stated in Introduction, an alternative proof of Lemma \ref{ub} is given by \cite{bow2}, 
see also \cite{so} where this lemma is proved by full geometric limit arguments along ideas in \cite{bow2}.

The other key lemma for the Lipschitz Model Theorem is replaced by the following 
lemma.
We will give a shorter proof of it.

\begin{lemma}\label{meridian}
Suppose that $\ve$ is any positive number and there exists a constant $L>0$ with $l_\rho(c)\leq 
L\mathrm{length}_{M_\nu[0]}(c)$ for any rectifiable curve $c$ in $M_\nu[0]$.
Then, there exists a constant $d_1$ depending only on $\ve,\ve_1, L$ such that, for any component $U(v)$ of $\cu$ 
with $|\omega_M(v)|> d_1$, $l_\rho(v)\leq \ve$.
\end{lemma}
\begin{proof}
Let $\lambda$ be the geodesic loop in $N_\rho$ freely homotopic to $f_0(v)$.
Suppose that $l_\rho(v)>\ve$.
If $\ve_1\mathrm{Im}(\omega_M(v))\geq n$, then there exist at least $n$ mutually non-homotopic pleated 
maps $p_j:F(\sigma_j)\to N_\rho$ such that each $p_j(\part F)$ contains $\lambda$ , where $F$ is a compact 
3-holed sphere.
Since $l_\rho(v)=\mathrm{length}_{N_\rho}(\lambda)>\ve$, all $p_j(F(\sigma_j)_{[\ve,\infty)})$ are contained in 
a uniformly bounded neighborhood of $\lambda$ in $N_{\rho[\ve,\infty)}$.
From this boundedness, we know that $\mathrm{Im}(\omega_M(v))$ is bounded by a constant $d$ depending 
only on $\ve$ and $\ve_1$.

Set $U(v)=U$ and let $m$ be the shortest geodesic in $\part U$ among all geodesics meeting a leaf $l$ of the foliation 
$\cf_v$ transversely in a single point.
The length of $m$ is at most $(d+1)\ve_1$.
If $m$ is a meridian of $U$, then $|\omega_M(v)|=\mathrm{length}_{\part U}(m)/\ve_1\leq d+1$.
Otherwise, $f_0|_m$ is homotopic to a cyclic covering $\eta:m\to \lambda$ whose 
degree is at most $L(d+1)\ve_1/\ve$. 
This means that the geometric intersection number $\alpha$ of $m$ with a meridian $m_0$ 
of $U$ is at most $L(d+1)\ve_1/\ve$.
Under a suitable choice of the orientations of $m$ and $l$, the homology class $[m_0]\in H_1(\part U,\zz)$ 
is represented by $[m]+\alpha[l]$ and hence
\begin{align*}
|\omega_M(v)|&=\frac{1}{\ve_1}\mathrm{length}_{\part U}(m_0)\leq \frac{1}{\ve_1}(\mathrm{length}_{\part U}(m)
+\alpha\mathrm{length}_{\part U}(l))\\
&\leq (d+1)\Bigl(1+\frac{L\ve_1}{\ve}\Bigr)=:d_1.
\end{align*}
This completes the proof.
\end{proof}

\subsection{Minsky's construction}\label{M_const}
Here we will review briefly how Minsky constructs the Lipschitz map.

Recall that, for each element $B$ of  the brick decomposition $\bsmcB$ of  $(S\times \wh \rr,\cv)$ defined 
in Subsection \ref{M_on_bu}, either $\xi(B)=0$ or 1 holds.
Let $\bsmcB_\part$ be the set of boundary bricks associated to elements of $\mathcal{GF}_+\cup \mathcal {GF}_-$.
In Subsection \ref{C_mm}, we re-embedded $M_\nu[0]=\bigcup (\bsmcB\cup \bsmcB_\part)$ into $S\times \rr$ so that 
$\cv$ is identified with $\cu\setminus \cu_{(\wh S\setminus S)}$, see Fig.\ \ref{f_model}.
For any element $B=F\times [a,b]$ of $\bsmcB$ with $\xi(B)=0$, let $F_B$ be the  horizontal core $F\times 
\bigl\{\frac{(b-a)}2\bigr\}$ of $B$ .
Then $f_0|_{F_B}:F_B\to N_\rho$ is homotopic to a pleated map $f_B$ such that, for each 
component $l$ of $\part F_B$, $f_B(l)$ is either a closed geodesic in $N_\rho$ or the ideal point of 
a parabolic cusp component of $N_{\rho(0,\ve_1)}$.
Fix a hyperbolic metric on $F$ isometric to $\Sg_{0,3}$.
By Length Upper Bounds Lemma (Lemma \ref{ub}), there exists a marking-preserving $K_1$-bi-Lipschitz map
$i_B:F\to F_B(\sigma_B)_{[\ve_0,\infty)}$ for some uniform constant $K_1\geq 1$, where $\ve_0$ is the 
constant given in Subsection \ref{hyp} and $\sigma_B$ is 
the hyperbolic structure on $F_B$ induced from that on $N_\rho$ via $f_B$.
Steps 1--6 in \cite[Section 10]{mi2} define a map $f_6:M_\nu\to N_\rho$ homotopic to $f_0$ and satisfying 
the following conditions.
\begin{enumerate}[\rm (a)]
\item
For any $B\in \bsmcB$ with $\xi(B)=0$, $f_6|_{F_B}=f_B\circ i_B$.
\item
For any vertex $v$ appeared in $H_\nu$ and satisfying $l_\rho(v)\leq \ve_1$, $f_6(U(v))$ is contained in a 
component of $N_{\rho(0,\ve_1)}$.
\item
For any $k\geq 0$, there exist uniform constants $L(k)\geq 1$ and $\ve(k)\in (0,\ve_0)$ such that the 
restriction $f_6|_{M_\nu[k]}$ is $L(k)$-Lipschitz and $f_6(M_\nu[k])\cap N_{\rho(0,\ve(k))}=\eset$.
\end{enumerate}

Applying Lemma \ref{meridian} to $f_6|_{M_\nu[0]}$ for $L=L(0)$, one can choose $k$ so that  
$l_\rho(v)\leq \delta$ for any $U(v)$ with $|\omega_M(v)|\geq k$, where $\delta>0$ is a constant less than $\ve_1/2$.
By the property (b), $f_6(U(v))$ is contained in a component $\bt(v)$ of $N_{\rho(0,\ve_1)}$.
Let $\bbt[k]$ be the union of all $\bt(v)$ with $|\omega_M(v)|\geq k$.
Lemma \ref{non_para} implies that $f_6$ defines a bijection between the components of $\cu[k]$ and $\bbt[k]$.
Here we may take the $k$ and hence $\delta$ so that $f_6(M_\nu[k])\cap \bt_{\delta}(v)=\eset$ for any 
component $U(v)$ of $\cu[k]$, where $\bt_{\delta}(v)$ is the component of $N_{\rho(0,\delta)}$ 
contained in $\bt(v)$.
Fixing such a $k$ and deforming $f_6$ by a homotopy whose support is contained in a neighborhood of 
$\cu[k]$ in $M_\nu$, we have a $K_7$-Lipschitz map $f_7$ with $f_7(\cu[k])=\bbt[k]$ and 
$f_7^{-1}(\bbt[k])= \cu[k]$.
Here we set $\ve_2=\ve(k)$ for the $k$.
A Lipschitz map $f=f_8$ is obtained by extending the definition of $f_7$ to $\cu_{\mathrm{p.c.}}$.
Minsky shows that the map $f$ is a proper degree one map satisfying the 
conditions of Theorem \ref{lm}.
The extension of $f$ to a $K$-bi-Lipschitz map $f':E_\nu\to E_N$ is proved by hyperbolic geometric arguments together with some differential geometric ones in \cite[Subsection 3.4]{mi2}.

\subsection{Additional properties of the Lipschitz map}\label{additional}

By the form (\ref{cu}) of $\cu$ and the property (i) of Theorem \ref{lm}, $\bbt[k]$ is represented as the disjoint union:
$$\bbt[k]=\bbt[k]_{\mathrm{int.}}\cup \bbt[k]_{\mathrm{g.f.}}\cup \bbt[k]_{\mathrm{p.c.}}.$$
We set $\wh g=(f\cup f'):M\!E_\nu\to N_\rho$ and consider the restriction
\begin{equation}\label{rest}
g=\wh g|_{M\!E_\nu[k]}:M\!E_\nu[k]\to N_\rho[k]:=N_\rho\setminus \bbt[k].
\end{equation}
Let $\ol{\cu[k]}$ be the closure of $\cu[k]$ in $M\!E_\nu[k]$.
Recall that a \emph{horizontal surface} in $M\!E_\nu[k]$ (resp.\ $M_\nu[k]$) is a connected surface $F$ in $S\times \{a\}$ (resp.\  $S\times \{a\}\cap M_\nu[k]$) for some $a\in \rr$ with $\Int F\cap \ol{\cu[k]}=\eset$ and $\part F\subset \ol{\cu[k]}$.

\begin{prop}\label{homeo}
For any horizontal surface $F$ in $M_\nu[k]$, the restriction $g|_F$ is properly homotopic an 
embedding $h:F\to N_\rho[k]$ which is uniformly bi-Lipschitz onto the embedded surface contained in the 
$1$-neighborhood of $g(F)$ in $N_\rho[k]$.
\end{prop}
\begin{proof}
Set $M\!E_\nu'=M\!E_\nu \setminus \cu_{(\wh S\sm S)}$ and $N_\rho'=N_\rho\setminus \bbt_{(\wh S\sm S)}$, 
where $\bbt_{(\wh S\sm S)}=\wh g(\cu_{(\wh S\sm S)}) \subset \bbt[k]_{\mathrm{p.c.}}$.
Then $M\!E_\nu[k]$ is a subset of $M\!E_\nu'$.
Suppose that $U_1,\dots,U_m$ are the components of $\cu[k]\setminus \cu_{\mathrm{p.c.}}$
such that the closure $\ol U_j$ in $M\!E_\nu[k]$ meets $\part F$ non-trivially.
Let denote $\wh g(U_j)=\bt_j$ and $U_1^m=U_1\cup\cdots\cup U_m$, $\bt_1^m=\bt_1\cup\cdots\cup \bt_m$.
Let $\{Q_1,\dots,Q_n\}$ be the set of components of $N_\rho'\setminus (g(F)\cup \bt_1^m)$ such that the closure 
of $Q_i$ in $N_\rho'$ is compact.
By Otal \cite{ot}, $\bt_1^m$ is unlinked in $N_\rho'$.
Hence, by \cite{fhs}, $g|_F$ is properly homotopic to an embedding in the union of the (closed) 
$1$-neighborhood $R$ of $g(F)$ in $N_\rho[k]$ and $Q_1,\dots,Q_n$.
Note that the union is also a compact set.
Suppose that $Q_1$ contains a component $\bt$ of $\bbt[k]$ and 
$U$ is the component of $\cu[k]$ with $\wh g(U)=\bt$.

There exists a properly embedded surface $S_0$ in $M\!E_\nu[k]$ with $S_0\supset F$ and such that the inclusion $S_0\subset 
M\!E_\nu'$ is a homotopy equivalence and  
one of the two components of $M\!E_\nu'\setminus S_0$, say $P$, is disjoint from $\ol U\cup U_1^m$. 
Fix a horizontal surface $S_1$ in $P$ sufficiently far away from $S_0$.
Then $\wh g|_{M\!E_\nu'\setminus (\ol U\cup U_1^m)}:M\!E_\nu'\setminus (\ol U\cup U_1^m)\to N_\rho'\setminus (\ol\bt\cup \bt_1^m)$ 
is properly homotopic to a map $\alpha$ such that $\alpha|_{S_1}$ is an embedding.
Let $P_0$ be the closure of the bounded component of $M\!E_\nu'\setminus S_0\cup S_1$, and let  
$A_i$ $(i=1,\dots,m)$ be a properly embedded vertical annulus in $P_0$ such that one of the components of $\part A_i$ is a longitude of $\part U_i$, see Fig.\ \ref{f_homeo}.
\begin{figure}[hbtp]
\centering
\scalebox{0.125}{\includegraphics[clip]{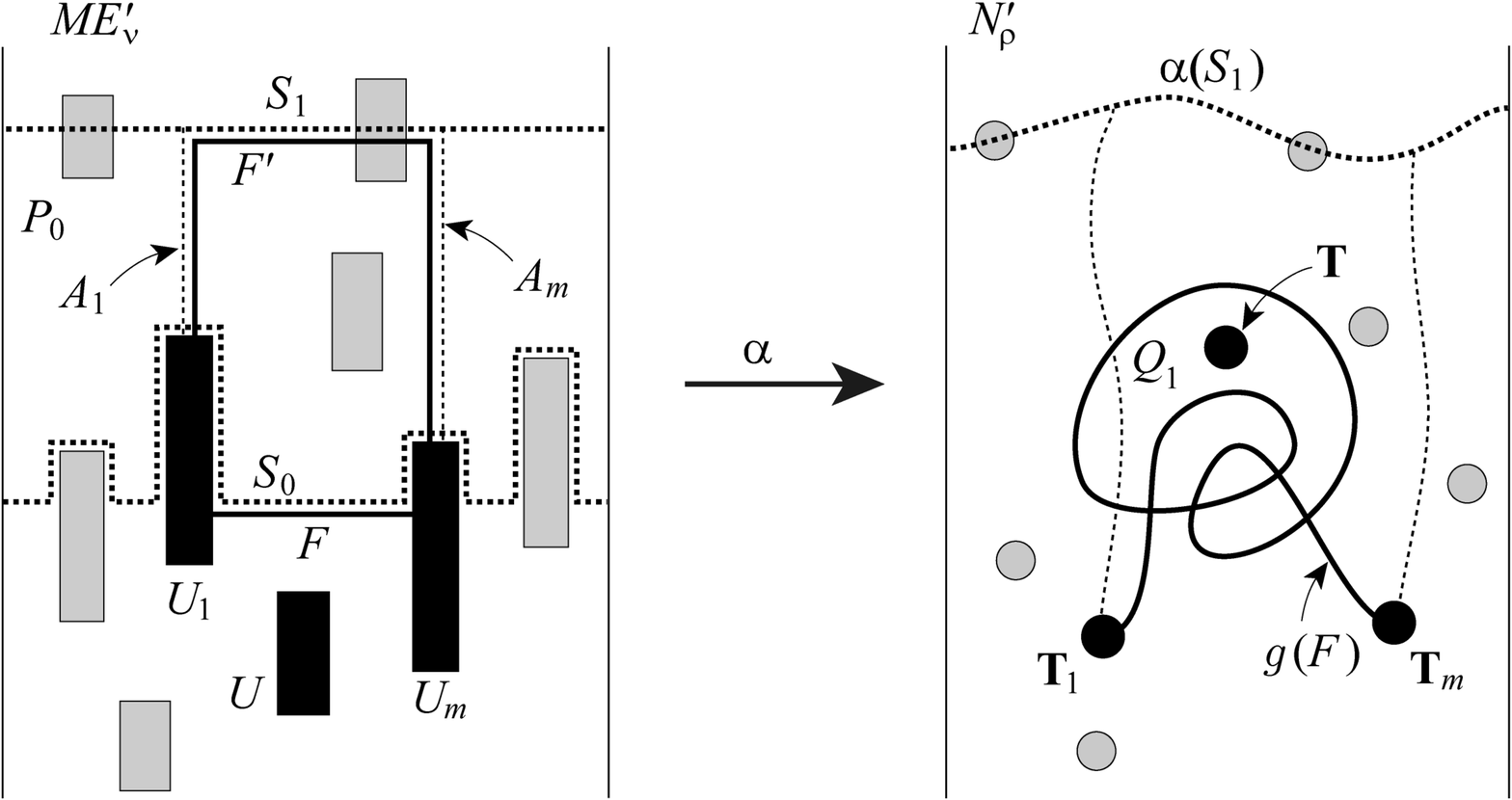}}
\caption{}
\label{f_homeo}
\end{figure}
If necessary deforming $\alpha$ by a proper homotopy again, we may assume that  
that the restriction $\alpha|_{A_1\cup\cdots\cup A_m}$ is also an embedding.
It follows from the fact that any two components of $\ol\bt\cup \bt_1^m$ are not parallel in 
$M\!E_\nu'$ and hence $\alpha|_{A_i}$ can not wind around any component of $\ol\bt\cup \bt_1^m$ homotopically essentially.
Thus $F$ is properly isotopic to a surface $F'$ in $M\!E_\nu'\setminus (\ol U\cup U_1^m)$ with $F'\subset S_1\cup  A_1\cup\cdots\cup A_m$ such that $\alpha|_{F'}$ is an embedding.
This shows that $g|_F$ is properly homotopic to an embedding in $N_\rho\setminus (\ol\bt\cup\bt_1^m)$.
Since $Q_2,\dots,Q_2$ are the components of $N_\rho'\setminus (g(F)\cup \bt_1^m\cup \ol\bt$) 
whose closures in $N_\rho' \setminus \ol\bt$ are compact, again 
by \cite{fhs} $g|_F$ is properly homotopic to an embedding in $R\cup (Q_2\cup\cdots\cup Q_m)$.
Repeating the same argument repeatedly, one can show that $g|_F$ is properly homotopic to an embedding $h$ 
in $R\cup Q_{u_1}\cup\cdots\cup Q_{u_a}\subset N_\rho[k]$, where $\{Q_{u_1},\dots,Q_{u_a}\}$ is the subset of 
$\{Q_1,\dots,Q_n\}$ with $Q_{u_j}\cap \bbt[k]=\eset$.

The uniform bi-Lipschitz property for a suitable embedding $h$ is derived easily from 
geometric limit arguments together with the uniform boundedness of the geometry on $R\cup Q_{u_1}\cup\cdots\cup Q_{u_a}$.
\end{proof}

A \emph{horizontal section} of $M\!E_\nu[k]$ is the union of horizontal surfaces of $M\!E_\nu[k]$ in the same level $S\times \{a\}$ for some $a\in\rr$. 
For any horizontal section $\Sigma$ of $M\!E_\nu[k]$, let $U_\Sg$ be the union of the components $U$ of $\cu[k]\setminus\wh\cu$ 
with $\part U\cap \Sg\neq \eset$.
Then, $\Sg$ separates $M\!E_\nu'\setminus U_\Sg$ into the $(+)$ and $(-)$-end components 
$P_+,P_-$.
By Proposition \ref{homeo}, $g:M\!E_\nu[k]\to N_\rho[k]$ is properly homotopic to a map $\beta$ such that 
$\beta|_\Sigma$ is an embedding.
The map $\beta$ is extended to a proper degree-one map $\wh \beta:M\!E_\nu'\to N_\rho'$.
The embedded surface $\wh\beta(\Sg)=\beta(\Sg)$ also separates $N_\rho'\setminus \bt_\Sg$ to the $(+)$ and $(-)$-end 
components $Q_+,Q_-$, where $\bt_\Sg=\wh \beta(U_\Sg)=\wh g(U_\Sg)$.
Since $\wh \beta$ defines a bijection between $\cu[k]$ and $\bbt[k]$, 
if a component $U$ of $\cu[k]$ is in $P_-$, then $\wh\beta(P_+)\cap \bt=\eset$ for $\bt=\wh\beta(U)=\wh g(U)$.
Since $\wh \beta(P_+)\supset Q_+$, $\bt$ is contained in $Q_-$.
Similarly, for any component $U$ of $\cu[k]\cap P_+$, $\wh g(U)$ is contained in $Q_+$.
This means that the pair  $(\Sg,\beta(\Sg))$ preserves the orders of $\cu[k]$ and $\bbt[k]$.

\begin{cor}\label{c_homeo}
The map $g$ of {\rm (\ref{rest})} is properly homotopic to a homeomorphism $g_0$.
\end{cor}

\begin{proof}
Let $\ch_0$ be a maximal set of horizontal surfaces in $M_\nu[k]$ such that any two elements of $\ch_0$ are 
not mutually parallel in $M_\nu [k]$.
From Proposition \ref{homeo} together with the order-preserving property of horizontal surfaces, we know that, for any $F_1,F_2\in \ch_0$, 
the restrictions $g|_{F_1}$ and $g|_{F_2}$ are properly homotopic to mutually disjoint embedded surfaces.
By \cite{fhs}, $g$ is properly homotopic to a map $g'$ such that $g'|_{\bigcup_{F\in \ch_0}F}$ is 
an embedding, where $g'(F)$ has the least area among all surfaces properly homotopic to $g(F)$ 
on a fixed Riemannian metric on $N_\rho[k]$ with respect to which $\part N_\rho[k]$ is locally convex.
By using standard arguments in 3-manifold topology (see for example \cite{wa, he}), one can prove 
that $g'$ is properly homotopic to a homeomorphism $g_0$ without moving $g'|_{\bigcup_{F\in \ch_0}F}$.
\end{proof}

In \cite[Proposition 3.1]{bow3}, this corollary is proved under more general settings.
We note that Corollary \ref{c_homeo} does not necessarily imply that $g_0$ is Lipschitz.
In fact, since we used the free boundary value problem of the minimal surface theory, we can not 
control the position of least area surfaces in $N_\rho [k]$.
For the proof of the bi-Lipschitz model theorem, we need to apply the fixed boundary value problem.

Let $F$ be any horizontal surface in $M_\nu[k]$.
Since  $F\cap \cu=F\cap (\cu\setminus \cu[k])$ and the geometries on all components of $\cu\setminus \cu[k]$ are uniformly bounded, one can show that any 
two horizontal surfaces in $M_\nu[k]$ with 
the same topological type are uniformly bi-Lipschitz up to marking.

\begin{remark}[Technical modifications on $g$]\label{mod_g}

Since the length of $g(l)$ is at most $K\ve_1$ for any boundary component $l$ 
of a horizontal surface in $M_\nu[k]$, we may assume by slightly modifying $g$ that  
the image $g(\part F)$ is a disjoint union of closed geodesics in $\part \bbt[k]$ for any 
horizontal surface $F$.

Let $U$ be a component of $\cu[k]\setminus \cu_{(\wh S\setminus S)}$ and $\mathbf{T}=\wh g(U)$.
If $\part U$ is a torus, then it consists of two horizontal annuli and two vertical annuli.
Otherwise, $\part U$ consists of one horizontal annulus and two vertical half-open annuli.
Let $L$ be the set of longitudes $l_i$ in $\part U$ corresponding to 
the boundary components of these horizontal annuli, $F(l_i)$ the horizontal surface in $M_\nu[k]$ with 
$\part F(l_i)\supset l_i$ and $A_j$ the horizontal annuli in $\part U$ with $\part A_j\subset L$.
Note that $L$ has either two or four components.
We say that $g|_L$ is \emph{well-ordered} if 
$g|_{\part U}:\part U\to \part\mathbf{T}$ is properly homotopic rel.\ $L$ to 
a homeomorphism.
Since the diameter of any horizontal surface $F$ in $M_\nu[k]$ is less than a uniform constant $\delta_0$, 
$\diam_{N_\rho[k]}(g(F))<K \delta_0$.
As in the proof of Proposition \ref{homeo},  there exists a proper homotopy for $g$ whose support 
consists of at most four components of uniformly bounded diameter and which moves $g$ to a map 
$\gamma$ such that $\gamma|_{\bigcup F(l_i)\,\cup\, \bigcup A_j}$ is an embedding into a small regular neighborhood of 
$g\bigl(\bigcup F(l_i)\cup \bigcup A_j\bigr)$ in $N_\rho[k]$, see Fig.\ \ref{f_well}.
\begin{figure}[hbtp]
\centering
\scalebox{0.125}{\includegraphics[clip]{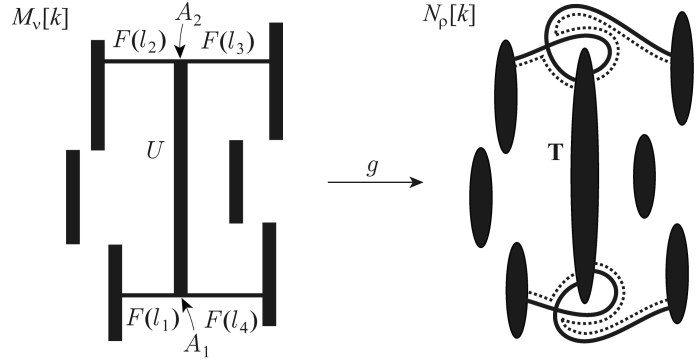}}
\caption{The dotted curves in the right side represent $\gamma\bigl(\bigcup F(l_i)\bigr)$.}
\label{f_well}
\end{figure}
Thus one can modify the Lipschitz map $g$ in a small neighborhood $\cn(\part U)$ of $\part U$ in $M_\nu[k]$ 
by a uniformly bounded-transferring homotopy so that $g^{\mathrm{new}}|_{\part U}=\gamma|_{\part U}$ 
and hence $g^{\mathrm{new}}|_L$ is well-ordered.
Here the homotopy being \emph{uniformly bounded-transferring} means that 
$\sup_{x\in M_\nu[k]}\{\dist_{N_\rho[k]}(g(x), \gamma(x))\}$ is less than a uniform constant.
The reason why we did not define $g^{\mathrm{new}}=\gamma$ totally in $M_\nu[k]$ is 
to do such a modification of $g$ on each component of $\part\cu[k]$ 
independently and simultaneously. 
The Lipschitz constant of $g^{\mathrm{new}}$ may be greater than the original constant, but still denoted by $K$.

Since $N_\rho[k]\subset N_{[\ve_2,\infty)}$ by Theorem \ref{lm}\,(i), 
modifying $g$ again if necessarily, one can suppose that $\dist_{\part\mathbf{T}}(\part_-A,\part_+A)\geq \ve_2/2$ for the closure $A$ of any component of $\part\mathbf{T}\setminus g(L)$.
\end{remark}

\subsection{Position of the images of horizontal surfaces}\label{intersection_number}

Let $\bsmcQ$ be the brick decomposition of $(M_\nu,\cu[k])$.
Note that $\bsmcQ$ may contain a brick $Q$ the form of which is either $F\times (-\infty,a]$ or 
$F\times [b,\infty)$ or $S\times \rr$.
For example, when $Q=F\times [b,\infty)$, $Q$ contains components of $\cu\setminus \cu[k]$ 
exiting the end of $Q$.
We say that a component of $\part_{\mathrm{hz}}Q$ contained in $S\times \rr$ (resp.\ in $S\times \{-\infty,\infty\}$) is a \emph{real front} (resp.\ an \emph{ideal front}) of $Q$.
Let $\sg(F)$ be the metric on a horizontal surface $F$ in $Q\in \bsmcQ$ induced from that on $M_\nu[k]$ and set $\dist(\sg(F),\sg(F'))=\dist_{\mathrm{Teich}(Q^S)}(\sg(F),\sg(F'))$.

Let $F,F'$ be horizontal surfaces in $Q\in \bsmcQ$.
Then $\dist_{M_\nu[k]}(F,F')$ is the length of a shortest arc $\alpha$ in $M_\nu[k]$ connecting 
$F$ with $F'$.
However, such an arc $\alpha$ may not be homotopic into $Q$ rel.\ $\part \alpha$.
So we consider the covering $p:\wt M_\nu[k]\to M_\nu[k]$ associated to $\fd(Q)\subset \fd(M_\nu[k])$ and set $\dist_{M_\nu[k];Q}(F,F')=\dist_{\wt M_\nu[k]}(\wt F,\wt F')$, 
where $\wt F$, $\wt F'$ are the lifts of $F$, $F'$ to $\wt M_\nu[k]$.
One can define $\dist_{N_\rho[k];Q}(g(F),g(F'))$ and $\diam_{N_\rho[k];Q}(g(B))$ for any brick $B$ in $Q$  
similarly by
using the covering $q:\wt N_\rho[k]\to N_\rho[k]$ associated to $g_*(\fd(Q))\subset \fd(N_\rho[k])$.
Note that, since $B$ is embedded in $Q$, $B$ and its lift to $\wt M_\nu[k]$ have the same diameter.

\begin{lemma}\label{l_tau}
For any $d>0$, there exists a uniform constant $\iota(d)$ satisfying the following conditions.
Let $F_j$ $(j=0,1)$ be horizontal surfaces in $Q\in \bsmcQ$ which contains simple non-contractible loops 
$w_j$ of length not greater than $\ve_1$.
If the geometric intersection number $i(w_0^S,w_1^S)\geq \iota(d)$, then $\dist_{N_\rho[k];Q}(g(F_0), g(F_1))\geq d$.
\end{lemma}
\begin{proof}
Form the construction of $M_\nu[k]$, we know that horizontal surfaces in $Q$ have uniformly bounded 
geometry up to marking.
Since moreover $N_\rho[k]\subset N_{\rho[\ve_2,\infty)}$, a geometric limit argument as in Example 
\ref{e_geom1} 
shows the existence of a uniform constant $\tau(d)>0$ such that, if $d(\sg(F_0),\sg(F_1))\geq \tau(d)$, 
then $\dist_{N_\rho[k];Q}(g(F_0), g(F_1))\geq d$.

Suppose here that $d(\sg(F_0),\sg(F_1))< \tau(d)$.
Then the length of a shortest loop $w_1'$ in $F_0$ freely homotopic to $w_1$ in $Q$ is bounded from above by a uniform 
constant $l(\tau(d))$.
Let $\alpha$ be any arc $\alpha$ in $F_0$ with $\part \alpha\subset w_0$ such that $\alpha$ is not homotopic in $F$ rel.\ $\part \alpha$ to an arc in $w_0$.
It is not hard to see that the length of $\alpha$ is not less than a uniform constant $\lambda>0$. 
Since $\lambda i(w_0^S,w_1^S)\leq \mathrm{length}_{F_0}(w_1')<l(\tau(d))$,  $\iota(d):=\lambda^{-1}l(\tau(d))$ is our desired uniform constant.
\end{proof}

For any brick $Q$ of $\bsmcQ$, we will define a new brick decomposition $\bsmcD_Q$ on $Q$.
From the definition of meridian coefficients in Subsection \ref{ss_meridian}, we know that, for  
any component $U$ of $\cu\setminus \cu[k]$, the diameter of $\part U$ is less than a uniform constant $\delta_1$.
We may assume that $\delta_1>1$.
Let $B$ be any brick of $Q$ such that at least one component $A$ of  $\partial_{\mathrm{vt}}B$ 
is contained in $\part U$ for some component $U$ of $\cu\setminus \cu[k]$.
Since any point of $B$ is connected with a point of $A$ along a path in a horizontal surface in $B$, 
the diameter of $B$ is at most $2\delta_0+\delta_1$.
By Lemma \ref{l_long_bricks}, either the diameter of $Q$ is less than $n_0(2\delta_0+\delta_1)$ or there exists a brick $C$ of $\wh\cb_a$ for some $a$ such that $Q^S$ is a compact core of $C^S$ and 
the compliment of $B_Q=C\cap Q$ in $Q$ consists of at most two components the closures $B_\alpha$ of which are bricks of diameter less than $n_0(2\delta_0+\delta_1)$. 
Hence $\diam_{N_\rho[k];Q}(g(B_\alpha))$ is less than the uniform constant $Kn_0(2\delta_0+\delta_1)=:\gamma_0$.
These $B_\alpha$ are called the \emph{complementary brick} of $B_Q$ in $Q$.
Since $\delta_1>1$, $\gamma_0>K(\delta_0+1)$.

According to \cite[Lemma 2.1]{mi1}, there exists a uniform constant $d_0=d_0(2\gamma_0)$ such that $d_{\cc(F)}(u,v)\geq d_0$ implies 
$i(u,v)\geq \iota(2\gamma_0)$ for any $u,v\in \cc_0(F)$, where $\iota(\cdot)$ is the uniform 
constant given in Lemma \ref{l_tau}. 
Let $g_C$ be the tight geodesic in $C^S$ defined in Subsection \ref{ss_hierarchies}.
Consider the subsequence $g_{B_Q}=\{v_i\}_{i\in I}$ of the tight geodesic $g_C$ consisting of entries 
$v_i$ with $A(v_i)\cap \Int B_Q\neq \eset$, where $I$ is an interval in $\mathbf{Z}$.
In the case of $\xi(Q)=1$, one can adjust $B_Q$ in $Q$ so that $\ca_{g_{B_Q}}\cap \part_\pm B_Q\neq 
\eset$ if $\part_\pm B_Q\neq \eset$.

Suppose that the cardinality $|I|$ of $I$ is greater than $2d_0$.
Then there exists a maximal subsequence $\{i_j\}_{j\in J}$ of $I=\{i\}$ with $d_0\leq i_{j+1}-j_j<2d_0$ and  
containing $\inf I$, $\sup I$ if they are bounded.
Consider horizontal surfaces $F_j$ $(j\in J)$ in $Q$ such that $F_j\subset B_Q$ and $F_j\cap A(v_{i_j})\neq \eset$ if $i_j\not\in \{\inf I,\sup I\}$ and $F_j=\part_-Q$ if $i_j=\inf I$, $F_j=\part_+Q$ if 
$i_j=\sup I$.
Let $\bsmcD_Q$ be the set of bricks $D_j$ in $Q$ with $\part_{\mathrm{hz}}D_j=F_j\cup F_{j+1}$.
In the case that $|I|\leq 2d_0$, we suppose that $\bsmcD_Q$ is the single point set $\{Q\}$.
We denote the union $\bigcup_{Q\in \bsmcQ}\bsmcD_Q$ by $\bsmcD$.

For any element of $D$ in $\bsmcD_Q$ with $\part_{\mathrm{hz}}D\cap \part_{\mathrm{hz}}Q=\eset$, 
if $\xi(D)>1$, then
$\part_- D$ and $\part_+D$ are connected by the union $R$ of at most $2d_0$ bricks in $D$ of diameter not greater than $2\delta_0+\delta_1$.
Since each horizontal surface $F'$ of $D$ meets $R$ non-trivially, the diameter of $D$ is less than 
$2d_0(2\delta_0+\delta_1)+2\delta_0=:\delta_2'$.
If $\xi(D)=1$, then $D$ contains at most $2d_0$ buffer bricks each of which is isometric to 
either $B_{0,4}$ or $B_{1,1}$.
Then one can retake the uniform constant $\delta_2'$ if necessary so that $\diam(D)<\delta_2'$ 
even if $\xi(D)=1$. 
In the case that $\part_{\mathrm{hz}}D\cap \part_{\mathrm{hz}}Q\neq \eset$, $D$ contains at most 
two complementary bricks $B_\alpha$.
Since $\diam(B_\alpha)<n_0(2\delta_0+\delta_1)$, the diameter of $D$ is less than $\delta_2'+2n_0(2\delta_0+\delta_1)=:\delta_2$.
It follows that $\delta_2$ is a uniform constant with 
\begin{equation}\label{eta_0}
\diam(D)< \delta_2\quad\mbox{for any $D\in \bsmcD$.}
\end{equation}
Similarly, each component of $\part_{\mathrm{vt}}D$ is an annulus of diameter less than $\delta_2$.

We say that a sequence of horizontal surfaces $\{Y_l\}_{l\in L}$ in $Q$ indexed by 
an interval in $\zz$ \emph{ranges in order} in $M_\nu[k]$ if $\wt Y_{l-1}$ and $\wt Y_{l+1}$ are contained 
in distinct components of $\wt M_\nu[k]\setminus \wt Y_l$ for any $\{l-1,l,l+1\}
\subset L$, where $\wt Y_u$ is the lift of $Y_u$ to the covering $p:\wt M_\nu[k]\to M_\nu[k]$ associated to 
$\fd(Q)\subset \fd(M_\nu[k])$.
The definition of $\{g(Y_l)\}_{l\in L}$ \emph{ranging in order} in $N_\rho[k]$ is defined 
similarly when $g(Y_l)\cap g(Y_{l+1})=\eset$ for any $\{l,l+1\}\subset I$.

\begin{lemma}\label{l_Y}
Let $Q$ be a element of $\bsmcQ$ such that $\bsmcD_Q$ has at least two elements.
Then, for the sequence $\{F_j\}_{j\in J}$ of horizontal surfaces in $Q$ as above, 
$\{g(F_j)\}$ ranges in order in $N_\rho[k]$ and, for any $j\in J$ and $n\in \nn$ with $F_{j+n}$ well defined,
\begin{equation}\label{eqn_gamma}
\dist_{N_\rho[k];Q}(g(F_j),g(F_{j+n}))\geq n\gamma_0.
\end{equation}
\end{lemma}

\begin{proof}
Set $F_j'=\part_\pm B_Q$ if $F_j=\part_\pm Q$ and $F_j'=F_j$ otherwise.
Both $F_j'\cap A(v_{i_j})$ and $F_{j+1}'\cap A(v_{j_{j+1}})$ contain simple non-contractible loops $w_1$, $w_2$ of length $\ve_1$, respectively.
Since $d_{\cc(Q^S)}(w_1^S,w_2^S)=i_1\geq d_0$, $i(w_1^S,w_2^S)\geq \iota(2\gamma_0)$.
By Lemma \ref{l_tau}, $\dist_{N_\rho[k];Q}(g(F_j'), g(F_{j+1}'))\geq 2\gamma_0$.
For the proof, we need to consider the case that $F_u'\neq F_u$ or $F_{u+1}'\neq F_{u+1}$ for some 
$u\in J$, say $F_u\neq F_u'$.
Then $F_{u+1}'=F_{u+1}$ since $\bsmcD_Q$ has at least two elements.
There exists a complementary brick $B_\alpha$ with $\part_{\mathrm{hz}}B_\alpha=
F_u\cup F_u'$.
Since $\diam_{N_\rho[k];Q}(g(B_\alpha))\leq \gamma_0$, $\dist_{N_\rho[k];Q}(g(F_u), g(F_{u+1}))\geq \gamma_0$.
It follows that $\dist_{N_\rho[k];Q}(g(F_j), g(F_{j+1}))\geq \gamma_0$ for any $j\in J$.

If $\{g(F_j'),g(F_{j+1}'),g(F_{j+2}')\}$ did not range in order in $N_\rho[k]$, then for some integer $a$ 
with $i_j\leq a\leq i_{j+2}$, there would exist 
horizontal surfaces $G_a$, $G_a'$ in $B_Q$ with $G_a\cap A(v_a)\neq \eset$, 
$\dist_{M_\nu[k];Q}(G_a,G_a')\leq 1$ and $g(G_a')\cap g(F_{j+b}')\neq \eset$, where 
$b=2$ if $i_j\leq a\leq i_{j+1}$ and $b=0$ if $i_{j+1}\leq a\leq i_{j+2}$.
Here $G_a'$ is taken to be equal to $G_a$ unless $\xi(Q)=1$ and $G_a'$ is in a buffer brick.
Since $d_{\cc(Q^S)}(v_a,v_{i_{j+b}})\geq d_0$, Lemma 
\ref{l_tau} would imply $\dist_{N_\rho[k];Q}(g(G_a), g(F_{j+b}'))\geq 2\gamma_0$.
On the other hand, since $g(G_a')\cap g(F_{j+b}')\neq \eset$,
\begin{align*}
\dist_{N_\rho[k];Q}(g(F_{j+b}'),g(G_a))&\leq \diam_{N_\rho[k];Q}(g(G_a'))+\dist_{N_\rho[k];Q}(g(G_a'), g(G_a))\\
&\leq K\delta_0+K<\gamma_0.
\end{align*}
This contradiction shows that $\{g(F_j'),g(F_{j+1}'),g(F_{j+2}')\}$ ranges in order in $N_\rho[k]$.
Since $\dist_{N_\rho[k];Q}(g(F_v'), g(F_{v+1}'))\geq 2\gamma_0$ for $v=j,j+1$, 
$\dist_{N_\rho[k];Q}(g(F_w), g(F_w'))\leq \gamma_0$ for $w=j,j+2$ and $F_{j+1}'=F_{j+1}$, it follows that 
$\{g(F_j),g(F_{j+1}),g(F_{j+2})\}$ also ranges in order and hence $\{g(F_j)\}$ does.
Then the inequality (\ref{l_Y}) is derived immediately from $\dist_{N_\rho[k];Q}(g(F_j), g(F_{j+1}))\geq \gamma_0$ for any $j$.
\end{proof}

For any component $U$ of $\cu[k]$, $\part U$ has the 
foliation  $\cf_U$  consisting of geodesic longitudes of length $\ve_1$.
By Remark \ref{mod_g}, the boundary $\part \bt$ of $\bt=\wh g(U)$ can have the foliation $\cg_U$ 
consisting of geodesic leaves 
such that $g(l)\in \cg_U$ for any leaf $l$ of $\cf_U$.
Thus $g|_{\part U}$ defines a $K$-Lipschitz map $\theta_U:\cf_U\to \cg_U$, where $\cf_U$ and 
$\cg_U$ have the metrics defined by the leaf distance in the Euclidean cylinders $\part U$ and $\part \bt$ respectively.
Any contractible component of $\cf_U$ or $\cg_U$ can be identified with an interval in $\rr$ as a metric space.
For any annulus $A$ in $\part U$ with geodesic boundary, the subfoliation of $\cf_U$ with the 
support $A$ is denoted by $\cf_A$.
When $A$ is vertical, for any $x\in \cf_A$, the horizontal surface in $M_\nu[k]$ which has a boundary component corresponding to $x$ is denoted by $F(x)$.
If $F(x)$ is a component of $\part_{\mathrm{hz}}D$ for some $D\in \bsmcD$, then $x$ is called a 
\emph{sectional point}.

\section{Geometric proof of the bi-Lipschitz model theorem}\label{bLip}

In this section, we will present a hyperbolic geometric proof of the bi-Lipschitz model theorem 
given in \cite{bcm}.

\begin{theorem}[Bi-Lipschitz Model Theorem]\label{blm}
There exist uniform constants $K'\geq 1,k>0$ such that there is a marking-preserving $K'$-bi-Lipschitz homeomorphism 
$\varphi:M\!E_\nu[k] \to N_\rho[k]$ which can be extended to a conformal homeomorphism from 
$\part_\infty M\!E_\nu$ to $\part_\infty N$.
\end{theorem}

For the proof, we need the following two lemmas.

\begin{lemma}\label{non_ret}
For any component $U$ of $\cu[k]$, let $A$  be a vertical component of $\part U$.
Then there exists a uniform constant $a_0$ such that, 
for any $x_0,x_1\in \cf_{A}$ with $\dist_{\cf_U}(x_0,x_1)\geq a_0$, 
$\dist_{\cg_U}(\theta_U(x_0),\theta_U(x_1))\geq K$.
\end{lemma}
\begin{proof}
Since each component of $\part_{\mathrm{vt}}D$ $(D\in \bsmcD)$ has diameter less than $\delta_2$, 
for any $x_i\in \cf_A$, there exists a sectional point $y_i\in \cf_A$ with $|x_i-y_i|\leq \delta_2/2$.
Since $\theta_U$ is $K$-Lipschitz, it suffices to show that there exists a uniform constant  $a_0$ with 
$|y_0-y_1|<a_0-\delta_2$ 
for any sectional points $y_0,y_1$ in $\cf_{A}$ with $|\theta_U(y_0)-\theta_U(y_1)|<K(\delta_2+1)$. 
We may assume that $y_0< y_1$ and $\theta_U(y_0)\leq \theta_U(y_1)$.
Consider the annulus $A'$ in $\part\mathbf{T}$ with $\cg_{A'}=[\theta_U(y_0),\theta_U(y_1)]$, where 
$\mathbf{T}=\wh g(U)$.
Set $X=g(F(y_0))\cup A'\cup g(F(y_1))$.
Since $\mathrm{diam}_{N_\rho[k];Q}(g(F(y_i)))\leq K\delta_0$ for $i=0,1$, $\diam(X)< K(2\delta_0+\delta_2+1)$.

Suppose that $g(F(y))\cap X$ is empty for some sectional point $y\in (y_0,y_1)$.
We may assume that $\theta_U(y)<\theta_U(y_0)$.
Since $g$ is properly homotopic to a homeomorphism $g_0$ by Corollary \ref{c_homeo}, one can exchange  the positions of $g(F(y))$ and $g(F(y_0))$ by a proper homotopy in $N_\rho[k]$.
If necessary modifying $g_0$ near $A$, we may assume that $g_0(\part_A F(y_0))=g(\part_A F(y_0))$, where 
$\part_A F(y_0)=F(y_0)\cap A$.
Since $g(F(y))\cap g(F(y_0))=\eset$ and $g_0(F(y_0))\cap g_0(F(y))=\eset$, by \cite{fhs} 
there exist properly embedded mutually disjoint surfaces $H_y$, $H_{y_0}$, $H_y'$ in $N_\rho[k]$ 
such that $g(F(y))$ is properly homotopic to $H_y$ rel.\ $g(\part_A F(y))$, 
both $g(F(y_0))$ and $g_0(F(y_0))$ to $H_{y_0}$ rel.\ $g(\part_A F(y_0))$,
and $g_0(F(y))$ to $H_y'$ rel.\ $g_0(\part_A F(y))$.
Since $H_y\cup H_y'$ excises from $N_\rho[k]$ a topological brick $B$ containing $H_{y_0}$ as a proper 
subsurface, $H_y$ is properly homotopic to $H_{y_0}$ in $N_\rho[k]$.
This implies that $F(y)$ and $F(y_0)$ are properly homotopic to each other in $M_\nu[k]$ and hence contained in the same brick $Q\in \bsmcQ$.

Let $Z$ be the set of sectional points $z$ of $\cf_{A\cap Q}$ with $z>y_0$.
By Lemma \ref{l_Y}, $\theta_U(z_+)<\theta_U(y_0)$ and $\theta_U(Z)$ is contained in the interval $[\theta_U(z_+),\theta_U(y_0))$, where $F(z_+)\subset \part_+Q$.
Since $\theta_U(z)<\theta_U(y_1)$ for any $z\in Z$, $y_1$ is not in $Z$.
If $y'$ is the smallest sectional point in $(z_+,y_1]$, then $g(F(y'))$ meets $X$ non-trivially.
Let $A''$ be the annulus in $\part \mathbf{T}$ with $\cg_{A''}=[\theta_U(z_+),\theta_U(y')]$ (or 
$[\theta_U(y'),\theta_U(z_+)]$) and $Y=A''\cup g(F(y'))$.
Since $\diam(A'')\leq K\delta_2$, $\diam(Y)\leq K(\delta_0+\delta_2)$.
If $g(F(z))\cap Y= \eset$ for $z\in Z$, then the positions of $g(F(y'))$ and $g(F(z))$ would be exchanged  
by proper homotopy in $N_\rho[k]$.
This contradicts that $y'\not\in Z$.
Hence $g(F(z))$ intersects $X'=\cn_{K(\delta_0+\delta_2)}(X,N_\rho[k])$.
It follows that $g(F(y))\cap X'\neq \eset$ for any sectional point $y$ in $[y_0,y_1]$.

The interval $[y_0,y_1]$ has at least $(y_1-y_0-\delta_2)/\delta_2$ sectional points $y_\alpha$.
Since the surfaces $F(y_\alpha)$ have mutually non-parallel simple non-contractible loops $l_\alpha$ with 
$\mathrm{length}_{N_\rho[k]}(g(l_\alpha))\leq K\ve_1$ and 
$\diam(X')$ is uniformly bounded, by a geometric limit argument as in Example \ref{e_geom1}, one can prove that $(y_1-y_0-\delta_2)/\delta_2$ is less than a uniform constant $m_0$.
Thus we have $|y_0-y_1|<a_0-\delta_2$ for $a_0:=(m_0+2)\delta_2$.
\end{proof}

For an interval $J$ in $\cf_U$, an interval $I$ in $\cg_U$ with $\part I=\theta_U(\part J)$ 
is the \emph{reduced image} of $J$ if $\theta_U|_J$ is homotopic rel.\ $\part J$ to a homeomorphism to $I$.

\begin{lemma}\label{rear}
There exist uniform constants $K_0,d_3$ such that $\theta_U$ is homotopic to a 
$K_0$-bi-Lipschitz map $\zeta_U:\cf_U\to \cg_U$  
such that $\dist_{\cg_U} (\theta_U(x),\zeta_U(x))<d_3$ for any $x\in \cf_U$.
\end{lemma}
\begin{proof}
Consider any component $U\in \cu[k]$ such that $\part U$ contains a vertical annulus component $A$ with $\mathrm{diam}_{\cf_U}(\cf_A)\geq a_0$.
Let $\{x_i\}$ be a sequence in $\cf_{A}$ with $a_0\leq x_{i+1}-x_i\leq 2a_0$ and 
$\cf_A=\bigcup_i J_i$, where $J_i=[x_i,x_{i+1}]$.
By Lemma \ref{non_ret}, the reduced image $I_i$ of $J_i$ satisfies 
\begin{equation}\label{cf_A}
K\leq \mathrm{diam}_{\cg_U}(I_i)\leq \mathrm{diam}_{\cg_U}(\theta_U(J_i))\leq 2Ka_0.
\end{equation}
Thus $\theta_U|_{\cf_A}:\cf_A\to \cg_U$ is homotopic to the map $\zeta_A:\cf_A\to \cg_U$ 
rel.\ $\{x_i\}$ such that, for any $J_i$, the restriction $\zeta_A|_{J_i}$ is an affine 
map onto $I_i$.
Then, by (\ref{cf_A}), $\dist_{\cg_U}(\theta_U(x),\zeta_A(x))<2Ka_0$ for any $x\in \cf_A$.
If $I_i\cap I_{i+1}\setminus \{x_{i+1}\}$ were not empty, then there would exist $z_i\in J_i$ and $z_{i+1}\in J_{i+1}$ 
with $\max\{x_{i+1}-z_i,z_{i+1}-x_{i+1}\}=a_0$ and $\theta_U(z_i)=\theta_U(z_{i+1})$.
Since $z_{i+1}-z_i\geq a_0$, this contradicts Lemma \ref{non_ret}. 
Thus, by (\ref{cf_A}), $\zeta_A$ is a uniformly bi-Lipschitz map onto an interval in $\cg_U$.

Let $A'$ be a horizontal component of $\part U$.
If $A'$ is not contained in a boundary brick in $\bsmcB_\part$, then $A'$ is isometric to 
$S^1(\ve_1)\times [0,1]$ as defined in Subsection \ref{M_on_bu} and hence $\mathrm{diam}_{\cf_U}(\cf_{A'})=1$.
By Remark \ref{mod_g}, the reduced image $I$ of $\cf_{A'}$ satisfies
$$\frac{\ve_2}2 \leq \mathrm{diam}_{\cg_U}(I)\leq \mathrm{diam}_{\cg_U}(\theta_U(\cf_{A'}))\leq K.$$
Thus $\theta_U|_{\cf_{A'}}:\cf_{A'}\to \cg_U$ is homotopic to a uniformly bi-Lipschitz map $\zeta_{A'}:\cf_{A'}\to I'\subset \cg_U$ 
rel.\ $\part \cf_{A'}$ by a uniformly bounded-transferring homotopy.
If $A'$ is contained in a boundary brick, then $\zeta_{A'}=\theta_U|_{A'}:A'\to \cg_U$ is already  
uniformly bi-Lipschitz onto the image by Theorem \ref{lm}\,(iii). 
The union $\zeta_U$ of these bi-Lipschitz maps is our desired map.
\end{proof}

\begin{proof}[Proof of Theorem \ref{blm}]
By Lemma \ref{rear}, there exists a uniform constant $K_1$ such that $g:M\!E_\nu[k]\to N_\rho[k]$ is properly 
homotopic to a $K_1$-Lipschitz map $g_1$ with 
$\dist_{N_\rho[k]}(g(x),g_1(x))\leq d_3+1$
for any $x\in M\!E_\nu[k]$  and 
such that the restriction 
$g_1|_{\part U}$ induces the $K_1$-bi-Lipschitz map $\zeta_U:\cf_U\to \cg_U$ for any component 
$U$ of $\cu[k]$, where the support of the homotopy is contained in 
a small collar neighborhood of $\part \cu[k]$ in $M\!E_\nu[k]$.
Here `$+1$' just means that $d_3+1$ is a constant strictly greater than $d_3$.
Since the original $g|_{E_\nu}:E_\nu\to E_N$ is uniformly bi-Lipschitz by Theorem \ref{lm}\,(iii), 
we may suppose that $g_1|_{E_\nu}$ is also a uniformly bi-Lipschitz map 
onto $E_N$.

Deform the metric on $N_\rho[k]$ in a small collar neighborhood of $\part N_\rho[k]$ so that $\part N_\rho[k]$ 
is locally convex but the sectional curvature of $N_\rho[k]$ is still pinched by $-1$ and some uniform constant 
$\kappa_0>0$.
For any critical horizontal surface $G_\alpha$ of $M\!E_\nu[k]$, let $H_\alpha$ be a surface in $N_\rho[k]$ 
which has the least area with respect to the modified metric on $N_\rho[k]$ among all 
surfaces properly homotopic to $g_1(G_\alpha)$ without moving their boundaries.
By Proposition \ref{homeo}, $g_1(G_\alpha)$ is properly homotopic to an embedding without moving the boundary.
By \cite{fhs}, $H_\alpha$ is also an embedded surface and $H_\alpha\cap H_\beta=\eset$ whenever  
$H_\alpha\neq H_\beta$.
Since the area of $G_\alpha$ is less than some uniform constant $A_0$, 
$\area(H_\alpha)\leq \area(g_1(G_\alpha))\leq K_1^2A_0$.
Since $N_\rho[k]\subset N_{\rho[\ve_2,\infty)}$ by Theorem \ref{lm}\,(i), the injectivity radius of $H_\alpha$ is 
not less than $\ve_2$.
Since moreover the intrinsic curvature of $H_\alpha$ at any point is at most $\kappa_0$, 
the diameter of $H_\alpha$ is less than a uniform constant.
As was seen in Example \ref{e_geom2} and Remark \ref{r_geom}, there exists a uniform constant $K_2>1$ such that 
$g_1$ is homotopic without moving $g_1|_{\part M\!E_\nu[k]}$ to a $K_2$-Lipschitz map $g_2$  
the restriction $g_2|_{G_\alpha}$ of which is a $K_2$-bi-Lipschitz map onto $H_\alpha$ for any $G_\alpha$.

Let $\{F_j\}$ be the sequence of horizontal surfaces in $Q\in \bsmcQ$ given in Lemma \ref{l_Y}.
Since $g_2$ is obtained from $g$ by a uniformly bounded-transferring homotopy, there exists 
a uniform constant $a_1\in \nn$ and a subsequence $\cy_Q=\{Y_l\}_{l\in L}$ of $\{F_j\}$ with $Y_l=F_{j_l}$ 
indexed by an interval $L$ in $\zz$ which satisfies the following conditions if $\bsmcD_Q$ contains at least $(a_1-1)$ bricks.
\begin{enumerate}[(i)]
\item
$Y_{\inf L}=\part_-Q$ and $Y_{\sup L}=\part_+Q$ if any.
\item
$j_{l+1}-j_l\leq a_1$ and 
$\dist_{N_\rho[k];Q}(g_2(Y_l),g_2(Y_{l+1}))\geq 3\gamma_0$ for any $\{l,l+1\}\subset L$.
\item
The sequence $\{g_2(Y_l)\}$ ranges 
in order from $g_2(\part_-Q)$ to $g_2(\part_+ Q)$ in $N_\rho[k]$.
\end{enumerate}

By (\ref{eta_0}) and (ii), $\dist_{N_\rho[k];Q}(g_2(Y_l),g_2(Y_{l+1}))\leq K_2\delta_2 a_1$.
Set $\mathcal{Y}=\bigcup_{Q\in \bsmcQ} \mathcal{Y}_Q$. 
Note that the $\gamma_0$-neighborhoods $\cn_{\gamma_0}(g_2(Y_u))$ of $g_2(Y_u)$ in 
$N_\rho[k]$ for $Y_u \in \mathcal{Y}_Q$ not in $\part_{\mathrm{hz}}Q$ are mutually disjoint and 
disjoint from the $\gamma_0$-neighborhood of $g_2(\part_{\mathrm{hz}}Q)$.
By Proposition \ref{homeo}, for any $Y_u \in \cy\setminus \bigcup_\alpha\{G_\alpha\}$, the restriction $g_2|Y_u :Y_u \to N_\rho[k]$ is properly homotopic to an embedding 
$h_u$ 
which is a $K_3$-bi-Lipschitz map onto a surface contained in $\cn_{\gamma_0}(g_2(Y_u))$ for some 
uniform constant $K_3\geq 1$.
Since the geometries on these embedded surfaces are uniformly bounded,   
there exists a uniform constant $K'\geq \max\{K_2,K_3\}$ as in Example \ref{e_geom2} such that $g_2$ 
is properly homotopic to a $K'$-bi-Lipschitz map $\varphi$ with  $\varphi|_{\bigcup_{\alpha}G_\alpha}=g_2|_{\bigcup_{\alpha}G_\alpha}$ 
and $\varphi|_{Y_u}=h_u$ 
for any $Y_u \in \mathcal{Y}\setminus \{G_\alpha\}$.
This completes the proof.
\end{proof}

It is well known that the bi-Lipschitz model theorem together with standard hyperbolic geometric 
arguments implies the Ending Lamination Conjecture.

\begin{theorem}[Ending Lamination Conjecture]\label{elc}
Let $N_\rho,N_{\rho'}$ be hyperbolic $3$-manifolds as in Subsection \ref{hyp} which have the same 
end invariant set $\nu$.
Then, any marking-preserving homeomorphism $f:N_\rho\to N_{\rho'}$ is properly homotopic to 
an isometry.
\end{theorem}
\begin{proof}
By Theorem \ref{blm}, there exist marking-preserving uniformly bi-Lipschitz maps $\varphi:M\!E_\nu[k]\to N_\rho[k]$ 
and $\varphi':M\!E_\nu[k]\to N_{\rho'}[k]$ which are extended to conformal homeomorphisms from 
$\part_\infty M\!E_\nu[k]$ to 
$\part_\infty N_\rho$ and $\part_\infty N_{\rho'}$ respectively.
One can furthermore extend $\varphi,\varphi'$ to uniformly bi-Lipschitz maps $\wh \varphi:M\!E_\nu\to N_\rho$ and 
$\wh\varphi':M\!E_\nu\to N_{\rho'}$ by using standard arguments of hyperbolic geometry, for example 
see \cite[Lemma 8.5]{bcm} or \cite[Lemma 5.8]{bow3}.
Then $\Phi=\wh \varphi'\circ \wh \varphi^{-1}:N_\rho\to N_{\rho'}$ is a marking-preserving bi-Lipschitz map.
The $\Phi$ is lifted to a bi-Lipschitz map $\widetilde \Phi:\hh^3\to \hh^3$  between the universal coverings, 
which is equivariant with respect to the covering transformations.
The map $\widetilde \Phi$ is extended to a quasi-conformal homeomorphism $\wt \Phi_{\part}$ 
on the Riemann sphere $\wh{\mathbf{C}}$ such that $\wt\Phi_{\part}|_{\Omega_\rho}$ is a conformal homeomorphism 
from $\Omega_\rho$ to $\Omega_{\rho'}$, where $\Omega_\rho$ is the domain of discontinuity of the Kleinian 
group $\rho(\fd(S))$.
By Sullivan's Rigidity Theorem \cite{su}, $\wt\Phi_{\part}$ is an equivariant conformal map on $\wh{\mathbf{C}}$ and hence 
extended to an 
equivariant isometry $\widetilde \Psi:\hh^3\to \hh^3$, which covers an isometry $\psi:N_\rho\to N_{\rho'}$ 
properly homotopic to $f$.
\end{proof}

\end{document}